\documentclass[11pt]{amsart}
\scrollmode
\usepackage{amsmath, amsthm, amssymb,amsfonts,mathtools}
\usepackage{verbatim,geometry,graphicx,multirow,array}
\usepackage{algpseudocode}
\usepackage{enumerate}

\usepackage{xy,psfrag}
\usepackage{datetime}

\input epsf

\def \Rtbt {{\R^{2\times 2}}}

\newcommand{\R}{{\mathbb R}}
\newcommand{\Rnxn}{{\R^{n\times n}}}
\newcommand{\bmat}[1]{ \begin{bmatrix}#1\end{bmatrix}}
\newcommand{\smat}[1]{ \left[\begin{smallmatrix} #1 \end{smallmatrix}\right]}
\newcommand{\cont}{{\mathcal C}}
\newcommand{\diag}{\operatorname{diag}}
\newcommand{\Rn}{{\R^n}}
\newcommand{\abs}[1]{\left| #1 \right|}
\newcommand{\itext}[1]{{\qquad\text{#1}\qquad}}
\renewcommand{\ss}{\scriptstyle}
\def\sddots{\mathinner{\raise3pt\vbox{\hbox{$\ss .$}}
    \raise1.5pt\hbox{$\ss .$}\hbox{$\ss .$}}}
\let\hat\widehat
\theoremstyle{plain}
\newtheorem{thm}{Theorem}[section]
\newtheorem{lem}[thm]{Lemma}
\newtheorem{cor}[thm]{Corollary}
\theoremstyle{definition}
\newtheorem{rem}[thm]{Remark}
\newtheorem{rems}[thm]{Remarks}
\newtheorem{defn}[thm]{Definition}
\newtheorem{exm}[thm]{Example}

\newcounter{algo}[section]
\renewcommand{\thealgo}{\thesection.\arabic{algo}}
\newcommand{\algo}[3]{\refstepcounter{algo}
\begin{center}\begin{figure*}[h!]
\framebox[\textwidth]{
\parbox{0.95\textwidth} {\vspace{\topsep}
{\bf Algorithm \thealgo : #2}\label{#1}\\
\vspace*{-\topsep} \mbox{ }\\
{#3} \vspace{\topsep} }}
\end{figure*}\end{center}}

\begin{document}

\title{Decompositions and coalescing eigenvalues of \\ symmetric definite pencils
depending on parameters}
\author[Dieci]{Luca Dieci}
\address{School of Mathematics, Georgia Institute of Technology,
Atlanta, GA 30332 U.S.A.}
\email{dieci@math.gatech.edu}
\author[Papini]{Alessandra Papini}
\address{Dept. of Industrial Engineering, Univ. of Florence, viale
G. Morgagni 40-44, 50134 Florence, Italy}
\email{alessandra.papini@unifi.it}
\author[Pugliese]{Alessandro Pugliese}
\address{Dept. of Mathematics, Univ. of Bari ``A. Moro,''
Via Orabona 4, 70125 Italy}
\email{alessandro.pugliese@uniba.it}


\begin{abstract}
In this work, we consider symmetric positive definite pencils depending on two
parameters.  That is, we are concerned with the generalized eigenvalue problem
$A(x)-\lambda B(x)$,
where $A$ and $B$ are symmetric matrix valued functions in $\Rnxn$,
smoothly depending on parameters $x\in \Omega\subset \R^2$;
further, $B$ is also positive definite.   %It is well understood that,
In general, the eigenvalues of this multiparameter problem will not be smooth, the lack
of smoothness resulting from eigenvalues being equal at some parameter values (conical
intersections).   We first give general theoretical results on the smoothness of eigenvalues
and eigenvectors for the present generalized eigenvalue problem, and hence for the
corresponding projections, and then 
%Indeed, our specific goal in the present work is to 
perform a numerical study of the statistical properties of coalescing eigenvalues for
pencils where $A$ and $B$ are either full or banded, for several bandwidths.
%Unlike the more standard GOE ensemble, 
Our numerical study will be performed with respect to a 
random matrix ensemble which respects %is more faithful to 
the underlying engineering problems motivating our study.
\end{abstract}

\begin{flushright} \small Version of \today $, \,\,\,$ \xxivtime \end{flushright}

\maketitle

\pagestyle{myheadings}
\thispagestyle{plain}
\markboth{Dieci, Papini and Pugliese}{Coalescing eigenvalues of definite pencils}

\section{Introduction}\label{sec1}
An important and well studied problem in structural engineering is the second order problem
\begin{equation}\label{MEstatic}
M\ddot y+D\dot y+Cy\ = \ F(t)\ ,
\end{equation}
where the matrices $M,\ D,\ C\ \in \Rnxn$ are all symmetric and positive definite and typically arise
from finite element discretization of beams' structures and the like, and
a main interest of mechanical engineering is to study the way a structure responds to
specific external solicitations (the forcing $F$ above), in particular to solicitations
taking place at specified sinusoidal forcing;  e.g., see \cite{SWF:WavePerLatt}.

There are at least two outstanding difficulties in dealing with \eqref{MEstatic}.
The first is that the problem typically depends, smoothly, on one or more parameters
(that is, the matrices $M,D,K$, do), and this parameter dependence should be accounted for,
both in the development of algorithms and in the theoretical implications of the dependence itself.
The second difficulty is that typically $n\gg 1$ and
directly dealing with the 2-nd order problem \eqref{MEstatic} is not computationally feasible.
For this reason, a widely adopted technique consists in {\bf projecting} \eqref{MEstatic} onto
a subspace spanned by a restricted set of eigenvectors $v$ of the generalized (frictionless) eigenproblem
\begin{equation}\label{MEstaticGenEv}
(M-\lambda C) \ v\ = \ 0\ .
\end{equation}
For example, in \cite{Soize3}
the projection is taken with respect to the $m$ eigenvectors $v_1,\dots, v_m$, associated
to the $m$ smallest eigenvalues of \eqref{MEstaticGenEv}, with $m\ll n$.  That is, writing
$V=\bmat{v_1 \dots v_m}\in \R^{n\times m}$, instead of \eqref{MEstatic} one considers
\begin{equation}\label{MEstatic2}
\tilde M \ddot z+\tilde D\dot z+\tilde C z\ = \ \tilde F(t)\ ,
\end{equation}
where $y=Vz$, $z\in \R^m$, and $\tilde M=V^TMV$, $\tilde D =V^TDV$, $\tilde C=V^TCV$, and recall that
the eigenvectors $V$ can be chosen so that $\tilde M$ and $\tilde C$ are diagonal.

\begin{rem}
We are interested in carrying out the above plan in the parameter dependent case; more precisely,
we will consider the eigendecomposition of the
generalized eigenproblem \eqref{MEstaticGenEv} when the matrices depend on two parameters.  See below.  
The motivation, as above, 
is being able to perform a dimension reduction by projecting into a desired dominant eigenspace.  However,
in this case, there is another very important aspect to consider: in general, the projection is well defined
only if there is a gap between the eigenvalues associated to the eigenspace onto which we are projecting, and the
other eigenvalues.  (Note that this difficulty is true regardless of the smoothness of the function of eigenvectors.)
For this reason, our {\bf specific emphasis} will be to locate parameter values where the
eigenvalues coalesce: these are (and will be) called {\bf conical intersections} and are the parameter values
where the projection is not well defined.
\end{rem}

A plan of our paper follows.  In the remainder of this introduction, we review basic results on theory and
techniques for the {\emph{static}} generalized eigenproblem (that is, when the given matrices do not depend
on parameters).  Section \ref{Theory} contains smoothness and periodicty results for the parameter dependent case.
In particular, we give smoothness results about square roots (and Cholesky factors) when the matrices are smooth functions of
several parameters, and specialized results when they are analytic function of one real parameter.  We give a
block-diagonalization result, and also
discuss the codimension of having equal eigenvalues, and finally give some periodicity results for the
1-parameter case.  All of these results are needed for the development in Section \ref{CIs} about detection
of parameter values where the eigenvalues of the generalized eigenproblem coalesce.  In Section \ref{algos},
we discuss algorithmic development for detection of coalescing eigenvalues.  Finally, 
in Section \ref{NumExs} we give a collection of results on locating conical intersections of 
random functions ensambles, and give evidence on the power law distribution in terms of the size
of the problem.  The concern of how to build appropriate random models is addressed as well.
Throughout this work, the norm is always the 2-norm.

\subsection{Model problem and eigenvalues continuity}
Motivated by the discussion above, in particular \eqref{MEstatic2}, 
the basic problem we will consider is the following:
\begin{equation}\label{GenEv}
\left[A(x)-\lambda B(x)\right]  v \ = \ 0\ ,
\end{equation}
where $x\in \R^p$ represents $p$ parameters varying in an open and connected subset $\Omega$ of $\R^p$,
and we are interested in the cases of $p=1,2$.  The functions $A$ and $B$ in $\Rnxn$ will always be
symmetric and $B$ will also be positive definite, a fact that we will indicate with $B\succ 0$.
Moreover, both $A$ and $B$ will be $\cont^k$
functions of the parameters with $k\ge 1$; for the case of one parameter, we will also give some
results in the case of $A$ and $B$ being real analytic functions of the parameter.

We recall that the eigenvalues of \eqref{GenEv} are roots of the characteristic polynomial
$$\det(A-\lambda B)\ = \ 0\ .$$
Note that the leading coefficient of this polynomial is given by
$\det(B)$ and thus it is not $0$, since $B$ is positive definite, and therefore the
polynomial is of exact degree $n$.
As a consequence, there are $n$ eigenvalues of \eqref{GenEv}
and it is well known (see below) that they are all real.  Furthermore, 
since the coefficients of the polynomial $\det(A-\lambda B)$ are as smooth as the entries of $A$
and $B$ and the roots of a polynomial of (exact) degree $n$ depend continuously on the coefficients,
then we observe that the $n$ eigenvalues of
\eqref{GenEv} can be labeled so to be {\bf continuous} functions of the parameter $x$.
In particular, we can label them in decreasing order:
$\lambda_1(x) \ge \lambda_2(x) \ge \dots \ge \lambda_n(x)$.

Numerical methods for the ``static'' problem (that is, when $A,B\in \Rnxn$ are given
constant matrices, not depending on parameters), $A=A^T,\ B=B^T\succ 0$ (cfr.
\eqref{MEstaticGenEv}) are quite well developed; see \cite{Tisseur} for a review.
In essence, the standard techniques pass either through taking the square root of $B$ or its
Cholesky factorization, the
latter technique being the method of choice in the numerical community (e.g., it is the
method implemented in {\tt Matlab}).  For convenience, we review these below for this static
problem:
\begin{equation}\label{StaticGenEv}
\left[A-\lambda B \right]  v \ = \ 0\ .
\end{equation}

\begin{itemize}
\item[(a)] {\emph{Square root}}.
It is always possible to reduce the problem \eqref{StaticGenEv} to
a standard eigenvalue problem, as follows:
$$(A-\mu B)v=0\iff B^{1/2}\bigl(B^{-1/2}AB^{-1/2}-\mu I\bigr) B^{1/2}v=0 $$
$$ \qquad \iff (\tilde A-\mu I)w=0\ ,\, w=B^{1/2}v\ ,$$
where $B^{1/2}$ is the unique symmetric positive definite square root of $B$,
and $\tilde A=B^{-1/2}AB^{-1/2}$.  
Clearly, from the eigenvalues/eigenvectors of this last problem, we can get those
of \eqref{StaticGenEv}.  Since $\tilde A=\tilde A^T$, we note that the eigenvalues of
\eqref{StaticGenEv} (and those of \eqref{GenEv} for any given value of $x$)
are real, as previously stated.
\item[(b)] {\emph{Cholesky}}.
Similarly, since $B$ is positive definite, it admits a Cholesky factorization
\begin{equation}\label{Cholesky}
B=LL^T\ ,\,\, L\quad \text{lower triangular} \ , 
\end{equation}
from which it is immediate to obtain
$$(A-\mu B)v=0\iff L \bigl(L^{-1}AL^{-T}-\mu I\bigr) L^T v=0 $$
$$ \qquad \iff (\tilde A-\mu I)w=0\ ,\, \tilde A=L^{-1}AL^{-T}\ ,\, w=L^{T}v\ ,$$
and again from the eigenvalues/eigenvectors of this last problem, we can get those
of \eqref{StaticGenEv}.  We note that the Cholesky factor is not unique, but it can be made
unique by fixing the signs of $L_{ii}$, the standard choice being $L_{ii}>0$.
In this work, we will always restrict to this choice $L_{ii}>0$.
\end{itemize}

The following simple result will come in handy later on.
\begin{lem}\label{UniqueEvectors}
For \eqref{StaticGenEv},
the eigenvector matrix $V\in \Rnxn$, $V=[v_1,\dots, v_n]$,
can be chosen so to satisfy the relation
\begin{equation}\label{B-ortho}
V^TBV\ = \ I\ .
\end{equation}
If the eigenvalues $\lambda_i$'s are distinct, then, for a given ordering of the eigenvalues,
the matrix $V$ in \eqref{B-ortho} is unique up to the sign of its columns.
\end{lem}
\begin{proof}
Regardless of having used the square root of $B$ or its Choleski factor,
we saw that $(A-\lambda B)v=0\iff  (\tilde A-\mu I)w=0$, with $\tilde A= B^{-1/2}AB^{-1/2}$ or 
$\tilde A= L^{-1}AL^{-T}$.
Since $\tilde A=\tilde A^T$, then $\tilde A$ has an orthogonal matrix of eigenvectors
$W$: $W^TW=I$, and thus $V=B^{-1/2}W$, or $V=L^{-T}W$, satisfies $V^TBV=I$.  
In case the eigenvalues $\lambda_i$'s are
distinct, then it is well understood that, for given ordering of the eigenvalues, orthogonal $W$ is unique up to
the sign of its columns; that is, if $W_1$ and $W_0$ are two orthogonal matrices giving the same ordered
eigendecomposition of $\tilde A$, we must have $W_1=W_0D$ with $D=\bmat{ \pm 1 & & \\  & \ddots & \\  & & \pm 1 }$.
Therefore, %calling  $V_{0,1}=B^{-1/2}W_{0,1}$, 
we will also have $V_1=V_0D$.
\end{proof}
\begin{cor}
With the notation of the proof of Lemma \ref{UniqueEvectors}, we also have
\begin{equation*}%\label{Unique2}
V_1^TBV_0=D \iff V_1=V_0D\ . 
\end{equation*}
\end{cor}
\begin{proof}
($\Leftarrow$) Since $V_0^TBV_0=I$, then obviously $DV_0^TBV_0=D$.\\
($\Rightarrow$) From $V_0^TBV_0=I$ and $V_1^TBV_0=D$, given invertibility of $V_0$ and $B$,
we get $V_1=V_0D$ at once.
\end{proof}

\vspace{.5cm}

\section{A collection of smoothness results}\label{Theory}

Here we give several results for the generalized eigenvalue problem \eqref{GenEv} that extend known
results from the standard eigenvalue problem (that is, $B=I$ in \eqref{GenEv}):
\begin{equation}\label{StandardEv}
\left[A(x)-\lambda I)\right]  v \ = \ 0\ .
\end{equation}
It is well known (e.g., see \cite{Kato, DieciEirola, DiPu1}) that even for this standard eigenvalue problem 
\eqref{StandardEv} the
eigenvalues/eigenvectors cannot be expected to inherit smoothness of $A$, unless eigenvalues are distinct.
In the case of 2 parameters, in general there is a total loss of smoothness
when the eigenvalues coalesce (e.g., take $A=\bmat{x_1 & x_2 \\ x_2 & -x_1}$), and even in 
the 1 parameter case there is a potential loss of smoothness of the eigenvectors when eigenvalues
coalesce.
Our goal in this section is to generalize these, and similar, results, for \eqref{GenEv}. 

At a high level, one may argue that our results follow from the fact that $B$ (being positive definite)
induces an inner product, and hence a related concept of orthogonality; see Definition \ref{Bortho};
and, as a consequence, with the appropriate modifications
with respect to this inner product, many results from the standard case (symmetric eigenproblem
and Euclidean inner product) should follow.
Yet, these ``modifications'' are both non-trivial and of theoretical interest; moreover, our results have
practical engineering implications (e.g., see the discussion on smoothness of projection in the Introduction).
Indeed, given the relevance in Engineering applications of the generalized eigenproblem, we believe that our study is
both needed and timely.
%important to collect these results in a systematic way, results that we have not found in the literature.

\begin{defn}[$B$-orthogonality]\label{Bortho}
Let $B\in\cont^k(\R^p,\R)$ be a symmetric positive definite matrix valued function of $p$ parameters.
Two vector valued functions $v(x),w(x)\in \Rn$, are called {\emph{$B$-orthogonal}} if
$v^TBw=0$ and further {\emph{$B$-orthonormal}} if $v^TBv=1$ and $w^TBw=1$, for all $x$.
\end{defn}

\subsection{Square root}\label{SmoothRoot}

Before proceeding, we point out the following simple but important result,
whose proof we give for completeness.  % henceforth will be tacitly assumed.
\begin{lem}\label{SqRootFtn}
Let $a\in \cont^k(\Omega, \R)$, $k\ge 0$ an integer,
be a strictly positive function of $p$ real parameters $x\in \Omega$, where
$\Omega$ is an open, bounded, and connected subset of $\R^p$, and let $a$ be continuous
and uniformly bounded in $\bar \Omega$: $a(x)< \alpha$, $\alpha$ finite, $\forall x\in \bar \Omega$.
Then, the function $\sqrt{a(x)}$, where $\sqrt{a(x)}$ is the unique positive
square root of $a$, is also a $\cont^k$ function of $x$.
Furthermore, if $a\in \cont^\omega(J, \R)$ is analytic in the parameter $x\in J$, 
where $J$ is an open and bounded interval of the real line, then so is its square root.
\end{lem}
\begin{proof}
The $\cont^k$-result follows from Theorem \ref{smoothness} below.

For the case of analytic function of one parameter, we recall that composition of two analytic
functions is analytic, and thus what we need to prove is that the function $\sqrt{x}$ is analytic,
for $x$ in some interval $(c,b)$ with $0<c<b<\infty$.
But this follows from the following argument.
\begin{itemize}
\item[(i)]
Write $\displaystyle{\sqrt{x}=\sqrt{\alpha}\ \sqrt{x/\alpha}=\sqrt{\alpha} e^{\frac{1}{2}\ln (x/\alpha)}}$.
\item[(ii)]
Now, let $y=x/\alpha$, where $\alpha>b$, so that $0<y<1$.  Then, we have
$$\ln (x/\alpha)\ = \ 2 \sum_{j=0}^\infty \frac{1}{2j+1}\left[\frac{(y-1)}{(y+1)}\right]^{2j+1}\ .$$
This series converges for $y>0$, which is the case.  Moreover, the power series of $1/(y+1)$ is
$$\frac{1}{y+1}\ = \ \sum_{j=0}^\infty (-1)^j y^j\ ,$$
and this series converges for $y<1$, which also holds true.
\item[(iii)]  Putting together
the expressions in points (i) and (ii), we obtain a series expansion of $\ln (x/\alpha)$ and hence obtain
its analyticity.  The end result then follows since the exponential is also an analytic function.
\end{itemize}
\end{proof}

Next, we first observe that it is easy to infer that the Cholesky factor of $B$ is as smooth as $B$ itself
(for functions of one parameter, this result, and the argument of proof are known; see \cite{ChernDieci}).
\begin{thm}\label{SmoothChol}
Let $B\in \cont^k(\Omega,\Rnxn)$ be symmetric positive definite for all $x\in \Omega$.  Then, 
its Cholesky factor $L$ in \eqref{Cholesky} with $L_{ii}>0$
is also a $\cont^k$ function for $x\in \Omega$.
Further, if $B\in \cont^\omega(J, \Rnxn)$ is analytic in the parameter $x\in J$, 
where $J$ is an open and bounded interval of the real line, then so is the
Cholesky factor.
\end{thm}
\begin{proof}
The proof is immediate.  Write $B=\bmat{b_{11} & c^T \\ c & \hat B}$ and let
$L_1=\bmat{\sqrt{b_{11}} & 0 \\ c/\sqrt{b_{11}} & I}$, so that $L_1^{-1}BL_1^{-T}=\bmat{1 & 0 \\ 0 & B_1}$, where
$B_1=\hat B - cc^T/b_{11}$.  Obviously, $B_1$ is symmetric, positive definite, and as smooth as $B$, and
the result follows using Lemma \ref{SqRootFtn}.
The analytic case also follows in the same way since in this case $B_1$ and $\sqrt{b_{11}}$ are analytic.
\end{proof}

The next question is if the square root $B^{1/2}$ is also a $\cont^k$, respectively $\cont^\omega$,
function. Below, we prove that the answer is yes.  

We begin by showing that a symmetric positive definite function $B$, smoothly depending 
on parameters $x\in \R^p$, has a unique symmetric positive definite square root $S$,
which depends continuously on $x$.  We note that
continuity of the square root function can be inferred from the general result 
\cite[Proposition 2.1]{vanHemmenAndo}, but for completeness we give a different and 
more constructive proof.

\begin{lem}\label{continuity}
Let $B\in \cont^k(\Omega,\Rnxn)$, where $k\ge 1$, where $\Omega$ is an open, bounded, connected subset of 
$\R^p$, $p\ge 1$, and let $B\in \cont(\bar \Omega, \Rnxn)$.
Further, let $B$ be symmetric positive definite, and uniformly bounded, for all $x \in \bar \Omega$:
$\sup_{x\in \bar \Omega}\|B(x)\| < \gamma $.
Then, there exists, unique, a symmetric positive definite square root $S(x)$, for any $x \in \Omega$.
Moreover, $S(x)$ is a continuous function of $x$.
\end{lem}
\begin{proof}
We begin by scaling $B$ so that it will have norm less than $1$.  Namely, we define the function
$C(x)=\frac{B(x)}{\gamma}$, with $\gamma >0$, so that $\|C\|<1$ in $\bar \Omega$;
note that $C$ is positive definite 
and as smooth as $B$.   Then, we consider the function
$$I+(C-I)\equiv I+Y\,, \,\, Y:=C-I\ .$$
Observe that $Y$ is also symmetric, but negative definite and its eigenvalues are $-1+\mu$, where $\mu$ are
the eigenvalues of $C$.  Therefore, the eigenvalues of $Y$ are all in $(-1,0)$, and thus those of
$I+Y$ are all positive, and $I+Y$ has a unique positive definite square root for any given value of the parameters $x$.

Next, consider the following series expansion:
\begin{equation}\label{Series}
\left[I+Y\right]^{1/2}\ = \ \ I-Y/2 +Y^2/8 +\sum_{n=3}^\infty (-1)^n Y^n \frac{(2n-3)!}{n!(n-2)! 2^{2n-2}}\ ,
\end{equation}
and observe that all terms are symmetric, and smooth functions of the parameters.  Further, for any given
parameters value $x$, and any associated (unit) eigenvector $v$ of $Y$ so that $v^TYv=\nu\in (-1,0)$, one gets
$v^T\left[I+Y\right]^{1/2}v=\sqrt{1+\nu}$ as defined by the right hand side of \eqref{Series} with $\nu$ replacing
$Y$ there.  Therefore, the right hand side defines a positive definite matrix, for any given value of the parameters.

Now, let $y=\max_{\bar \Omega}\|Y\|$, and note that this gives $0<y<1$, and consider the numerical series
$$ 1+ y/2 + y^2/8 +\sum_{n=3}^\infty y^n \frac{(2n-3)!}{n!(n-2)!2^{2n-2}}\ .$$
By the ratio test, this numerical series converges if $y<1$, which is the case.  Therefore, 
from the Weierstrass M-test, we conclude that the series in \eqref{Series} converges uniformly.
As a consequence, the sum of the series is a continuous function of the parameters.
Finally, we observe that $(I+Y)^{1/2}=C^{1/2}$, and from $C(x)=\frac{A(x)}{\gamma}$,
we get that also $A^{1/2}=\sqrt{\gamma} {C}^{1/2}$ is a continuous function of the parameters.
\end{proof}

Using Lemma \ref{continuity}, we can get the result on smoothness.

\begin{thm}\label{smoothness}
With the same notation as in Lemma \ref{continuity}, 
the unique positive definite square root $B^{1/2}$ of the positive definite function $B\in \cont^k(\Omega, \Rnxn)$
is also a $\cont^k$ function. 
\end{thm}
\begin{proof}
Let $S=B^{1/2}$ and use that $S^2=B$.  We know that $S(x)$ is continuous, and that $B(x)$ is smooth.
Next, we define the first partial derivatives from formally differentiating the relation $S^2=B$.  That is,
consider
\begin{equation}\label{Lyap}
B_{x_i}=X_iS+SX_i\,, \quad i=1,\dots, p\ .
\end{equation}
The linear systems given by the Lyapunov equations in \eqref{Lyap} are uniquely solvable, since $S$ is positive definite
and thus invertible.  Now, the unique solution of an invertible linear system $Cz=b$ with $C$ and $b$ continuously depending
on parameters, obviously defines a continuous solution $z$, from which we conclude that the unique solutions 
$X_i$ of \eqref{Lyap} are continuous functions of the parameters $x$.  Finally, we observe that
$S_{x_i}=X_i$, $i=1,\dots, p$.

At this point, we can look at higher derivatives.  We see the situation for the second derivatives, from which the general
argument will be evident.

Rewrite \eqref{Lyap}
\begin{equation*}
B_{x_i}=S_{x_i}S+SS_{x_i}\,, \quad i=1,\dots, p\ ,
\end{equation*}
and consider the second partial derivatives from formally differentiating this relation.
We get:
\begin{equation}\label{Lyap2}
B_{x_ix_j}=X_{ij}S+S_{x_i}S_{x_j}+S_{x_j}S_{x_i}+SX_{ij}\,, 
\quad i,j=1,\dots, p\ .
\end{equation}
Rearranging terms in \eqref{Lyap2}, we obtain
$$B_{x_ix_j}-S_{x_i}S_{x_j}-S_{x_j}S_{x_i}=X_{ij}S+SX_{ij}\,, $$
which is again uniquely solvable and gives a continuous solution $X_{ij}$, and we observe that
$S_{x_ix_j}=X_{ij}$.  Finally, observe that, from the left-hand-side of \eqref{Lyap2}, we get
that $X_{ij}=X_{ji}$, that is $S_{x_ix_j}= S_{x_jx_i}$, and thus
the order of differentiation of the second partial derivatives does not matter. 

Continuing to formally differentiate, we obtain continuous higher derivatives and the result
follows.
\end{proof}

Finally, we specialize Theorem \ref{smoothness} to the case of $B$ analytic.

\begin{thm}\label{SqRootAna}
Let $B\in \cont^{\omega}(\R,\Rnxn)$, symmetric and positive definite for all $x$.  Then,
the unique positive definite square root $B^{1/2}$ is analytic in $x$ as well.
\end{thm}
\begin{proof}
The proof rests on a fundamental theorem of Kato, see \cite{Kato},
whereby an analytic Hermitian function admits an analytic eigendecomposition.  Thus, we can write
$B(x)=Q(x)D(x)Q^T(x)$ where $Q$ and $D$ are analytic, $Q$ is orthogonal, and $D$ is diagonal with
$D_{ii}(x)>0$ (we note that the eigenvalues in $D$ are not necessarily ordered).
Then, we have $B^{1/2}(x)=Q(x)D^{1/2}(x)Q^T(x)$, where $D^{1/2}(x)=\diag\left(\sqrt{D_{ii}(x)}\ ,\,\
i=1,\dots, n\right)$.  The result now follows from Lemma \ref{SqRootFtn}.
\end{proof}

\begin{rem}
Recalling that the positive definite square root of a positive definite matrix is unique,
from the numerical point of view we infer that --in principle-- any desired
algebraic technique can be used to compute the square root of $B$.
\end{rem}

Next, we look at a general block-diagonalization result for the parameter dependent generalized
eigenproblem, specializing a result given by Hsieh-Sibuya and Gingold (see \cite{HsiehSibuya} and
\cite{Gingold}) for the standard
eigenvalue case.  Then, we give more refined results for the case of one parameter, and
further specialize some results to the case of periodic pencils.  All of these results will form
the justification for our algorithms to locate conical intersections.

\subsection{General block diagonalization results}
In order to simplify the problem we consider, the following result is quite useful.
It highlights that the correct transformations for the pencil under study are ``inertia transformations''.

Since our interest in this work, for reasons which will be clarified below, is for 
the case where $A$ and $B$ depend on two (real) parameters, this is the case on which we focus in 
the theorem below.

\begin{thm}[Block-Diagonalization]\label{BlockDiag}
Let $R$ be a closed rectangular region in $\R^2$.
Let $A=A^T\in \cont^k(R,\Rnxn)$, $B=B^T\succ 0 \in \cont^k(R,\Rnxn)$, $k\ge 0$, and 
suppose that the eigenvalues of the pencil $(A, B)$ %(that is the values $\lambda$ such
%that $\det(A--\lambda B)=0$) 
can be labeled so that they belong to
two disjoint sets for all $x\in R$: $\lambda_1(x),\dots,\lambda_p(x)$ in
$\Lambda_1(x)$ and $\lambda_{p+1}(x),\dots, \lambda_n(x)$ in $\Lambda_2(x)$,
$\Lambda_1(x)\cap \Lambda_2(x) =\emptyset\ ,\, \forall x\in R$.  Then, there exists
%$V\in \cont^k(R,\Rnxn)$, $B$-orthogonal, such that
%$$V^{T}(x)A(x)V(x)=\bmat{ A_1(x) & 0 \\ 0 &  A_2(x)} \ ,\, 
%V^{T}(x)B(x)V(x)=\bmat{ I_p & 0 \\ 0 &  I_{n-p}(x)} \ ,\, 
%\forall x\in R\ ,$$
$V\in \cont^k(R,\Rnxn)$, invertible, such that
$$V^{T}(x)A(x)V(x)=\bmat{ A_1(x) & 0 \\ 0 &  A_2(x)} \ ,\, 
V^{T}(x)B(x)V(x)=\bmat{ B_1(x) & 0 \\ 0 &  B_2(x)} \ ,\, 
\forall x\in R\ ,$$
with $A_1=A_1^T, B_1=B_1^T\succ 0 \in \cont^k(R,\R^{p\times p})$,
and $A_2=A_2^T, B_2=B_2^T\succ 0 \in \cont^k(R,\R^{(n-p)\times (n-p)})$, 
so that the eigenvalues of the pencil $(A_1,B_1)$ are
those in $\Lambda_1$, and the eigenvalues of the pencil $(A_2,B_2)$
are those in $\Lambda_2$, for all $x\in R$.  Furthermore, the function $V$ can
be chosen to be $B$-orthogonal ($V^TBV=I$ for all $x$)..
\end{thm}
\begin{proof}
We show directly that the transformation $V$ can be chosen so that $V^TBV=I$, from
which the general result will follow.

One way to proceed is by using the unique smooth positive definite smooth square root of $B$, $B^{1/2}$, so that
the eigenvalues of the pencil are the same as those of the standard eigenvalue problem
with function $\tilde A=B^{-1/2}AB^{-1/2}$.  Because of Theorem \ref{smoothness}, the function $\tilde A$ is as
smooth as $A$ and it is clearly symmetric.  Therefore, from the cited results in \cite{HsiehSibuya, Gingold}, we
have that there exists smooth, orthogonal, $W$ such that $W^T\tilde AW=\bmat{A_1 & 0 \\ 0 & A_2}$,
with the eigenvalues of $A_i$ being those in $\Lambda_i$, and $A_i=A_i^T$, $i=1,2$.
Now we just take $V=B^{-1/2}W$.
\end{proof}
We notice that the function $V$ of Theorem \ref{BlockDiag} is clearly not unique, not even if we select one for
which $V^TBV=I$.

The block diagonalization result Theorem \ref{BlockDiag} can easily be extended to several blocks. 
In the case of $n$ distinct eigenvalues, one ends up with a full diagonalization.  
(Of course, having distinct eigenvalues is a sufficient, but not necessary, condition).
Because of its relevance in what follows, we give this fact as a separate result, with proof, in
the next subsection.  

\subsection{One parameter case: smoothness}
We say that the pencil -equivalently,
the generalized eigenproblem-- is {\emph{diagonalizable}} if
there are $n$ linearly independent eigenvectors $v_1,\dots, v_n$, associated to the eigenvalues
$\lambda_1, \dots, \lambda_n$.
Assembling these eigenvectors in a matrix $V=\bmat{v_1 & v_2 & \cdots v_n}$, 
and the associated eigenvalues along the diagonal of a matrix $\Lambda$,
$\Lambda=\diag(\lambda_1, \dots, \lambda_n)$, then we express the condition of diagonalizability in matrix form as
\begin{equation}\label{Diagonal}
AV\ = \ BV\Lambda \ .
\end{equation}

\begin{thm}\label{Smoothness}
Let $A=A^T\in \cont^k(\R,\Rnxn)$, $B=B^T\succ 0\in \cont^k(\R,\Rnxn)$, and assume that the eigenvalues
of \eqref{GenEv} are distinct for all $t$.  Then, the eigenvalues can be chosen to be $\cont^k$ functions
of $t$.  Moreover, we can also choose the corresponding eigenvector function $V$
to be a $\cont^k$ function of $t$ and to satisfy the relation $V^TBV=I$, for all $t$.
\end{thm}
\begin{proof}
The proof puts together known results, and we show an argument using the square root of $B$.

From Theorem \ref{SmoothRoot}, we know that the unique positive definite square root of $B$, call it $B^{1/2}$
is as smooth as $B$.  Then, we rewrite
$$AV\ = \ BV\Lambda \iff B^{-1/2} A B^{-1/2} (B^{1/2}V) \ = \ (B^{1/2}V)\Lambda \ . $$
Now, let $\tilde A=B^{-1/2}AB^{-1/2}$ and $W=B^{1/2}V$, and observe that 
$\tilde A=\tilde A^T\in \cont^k(\R,\Rnxn)$ and so we are left to show that we can choose $W$ to be a smooth
function of orthogonal eigenvectors of $\tilde A$, from which the result will follow.  But, obviously
the eigenvalues of $\tilde A$ are the same as those of $A$, and under the assumption of having distinct
eigenvalues it is known from \cite[Proposition 2.4]{DieciEirola} 
that the eigenvalues can be chosen smooth, and $W$ can be chosen smooth and
orthogonal, from which the result follows (note that $V^TBV=W^TW=I$).
\end{proof}

We can further refine the eigendecomposition result Theorem \ref{Smoothness}, 
even allowing for coalescing eigenvalues obtaining the following results about smoothness of
the eigenvalues/eigenvectors of the generalized eigenproblem \eqref{GenEv}.

\begin{thm}\label{OneParaSmooth}
Let $A, B \in \cont^k(J, \Rnxn)$, $k\ge 1$, and
$A=A^T$, $B=B^T\succ 0$ for all $x\in J$, where $J$ is some interval of the real line.
\begin{itemize}
\item[(i)] (Finite order of coalescing)
Suppose that the continuous eigenvalues
$\lambda_1, \dots, \lambda_{n}$, satisfy
\begin{equation*}%\label{SigmadCond}
\liminf_{\tau\to 0}\frac{\abs{\lambda_i(x+\tau)-\lambda_j(x+\tau)}}
{\abs{\tau^e}}\in (0, \infty]\ ,
\end{equation*}
for some $e\le k$ and for all $x\in J$ and $i\not= j$.  Then, there exists a
$B$-orthogonal function of eigenvectors $V\in \cont^{k-e}(J,\Rnxn)$.  The 
eigenvalues can be labeled so to be $\cont^k$ functions.
\item[(ii)] (Analytic case) 
Moreover, if $A, B \in \cont^\omega$, then the eigenvalues can be labeled so to be analytic
functions, and there is an associated $B$-orthogonal analytic function of eigenvectors.  
\end{itemize}
\end{thm}
\begin{proof}
As above, we reduce the problem
to that of a symmetric eigenproblem $(\tilde A-\lambda I)w=0$, with smooth, respectively analytic,
function $\tilde A$.  At this point, the stated
results are a direct application of known results in the literature for
symmetric functions of 1 parameter.  See \cite[Theorems 3.3 and 3.4]{DieciEirola}
for statements (i) and see \cite{Kato} for statement (ii).
\end{proof}

\begin{rem}\label{Isospectral}
The use of the square root of $B$ in the proof above is not necessary, and other possibilities
exist.  For example, using the Cholesky factor of $B$: $B=LL^T$, where $L$ is lower triangular
with positive diagonal; recall that, from Theorem \ref{SmoothChol} we know that $L$ is as smooth as $B$.  Using
this, we get 
$$ AV\ = \ BV\Lambda \iff L^{-1} A L^{-T} (L^{T}V) \ = \ (L^T V)\Lambda \iff \hat A W_C=W_C\Lambda$$ 
with $\hat A=L^{-1}AL^{-T}$, and $W_C= L^{T}V$.  As before, $\hat A$ is smooth and symmetric and $W_C$
is smooth and orthogonal.  This leads to an interesting consequence.
Assume that the distinct eigenvalues are
arranged along the diagonal of $\Lambda$ in a fixed way, say in increasing fashion, for both the
eigendecompositions of $\tilde A$ and of $\hat A$.  Then, the functions $\tilde A$ and $\hat A$ are
two symmetric isospectral functions, and are orthogonally similar.  
Indeed, let us call $W_S$ the orthogonal factor of $\tilde A$ (that is using the square root of $B$) 
and call $W_C$ the orthogonal factor of $\hat A$ (that is, using the Cholesky factor of $B$).  Then:
$\Lambda=W_S^T\tilde AW_S=W_C^T\hat AW_C$ from which we obtain that 
$\hat A=(W_SW_C^T)^T\tilde A (W_SW_C^T)$.  It is important to stress that, in spite of the differences
in the orthogonal factors of $\tilde A$ and $\hat A$, the end result on $V$ is essentially unique;
see Corollary \ref{UniqueV} below.
\end{rem}

Next is the uniquess result for $V$.

\begin{cor}\label{UniqueV}
Under the assumptions of Theorem \ref{Smoothness}, call $V$ a smooth function of eigenvectors 
satisfying $V^TBV=I$, and rendering a certain ordering for the diagonal of $\Lambda$.
Such $V$ is unique, and any other possible (smooth) function of
eigenvectors yielding the same ordering of eigenvalues is obtained from $V$ by sign changes of $V$'s columns.
\end{cor}
\begin{proof}
Since the eigenvalues are distinct, then the eigenvectors are uniquely determined up to scaling.
In other words, the only freedom in specifying $V$ is given by $V\to VS$ where $S=\diag(s_i,\, i=1,\dots, n)$
with $s_i\ne 0$.  By requiring that $V^TBV=I$, we get that we must have $S^2=I$, that is $s_i^2=1$, $i=1,\dots, n$,
as claimed.
\end{proof}

\subsubsection{Differential Equations for the factors}
Our goal in this section is to derive differential equations satisfied by the smooth 
factors $V$ and $\Lambda$, under the assumption of distinct eigenvalues. So doing, we will generalize
known results in \cite{DieciEirola} for the standard eigenproblem (i.e., when
$B=I$).  As it turns out, the generalization is not entirely trivial. 

Consider \eqref{GenEv}, with 
$A=A^T\in \cont^k(\R^+,\Rnxn)$, $B=B^T\succ 0\in \cont^k(\R^+,\Rnxn)$, and assume that the eigenvalues
of \eqref{GenEv} are distinct for all $t$.  As seen in Theorem \ref{Smoothness}, we can choose 
$V$ and $\Lambda$ smooth as well satisfying \eqref{Diagonal}
and $V$ satisfies the relation $V^TBV=I$, for all $t$.

As seen in Corollary \ref{UniqueV}, we must fix a choice for
$V$.  So, suppose we have an eigendecomposition at $t=0$, that is
we have $V_0$ and $\Lambda_0$ so that
$$A(0)V_0\ = \ B(0)V_0\Lambda_0 \ , \,\, V_0^TB(0)V_0=I\ . $$
We want to obtain differential equations satisfied by the factors $V$ and $\Lambda$ for all $t\ge 0$,
satisfying the initial condition $V(0)=V_0$ and $\Lambda(0)=\Lambda_0$. 

Since the factors are smooth, we can formally differentiate the two relations 
\begin{equation}\label{want}
\text{(a)}\quad AV-BV\Lambda=0 \itext{and} \text{(b)} \quad V^TBV=I\ .
\end{equation}
Differentiation of \eqref{want}-(a) gives
$$V^TA\dot V+\dot V^TBV\Lambda = \dot \Lambda -V^T\dot A V\ ,$$
from which using $AV=BV\Lambda$, and hence $V^TA=\Lambda V^TB$, we obtain
\begin{equation}\label{Step1}\begin{split}
\dot \Lambda -V^T\dot A V  & = \Lambda(V^TB\dot V)+(\dot V^TBV)\Lambda \quad \text{or} \\ 
\dot \Lambda -V^T\dot A V  & = \Lambda(V^TB\dot V)+(V^TB\dot V)^T \Lambda \ .
\end{split}\end{equation}
Now, using the structure of $\Lambda$ (diagonal) we observe that relatively to the
diagonal entries we have (using that the diagonals of $V^TB\dot V$ and of $(V^TB\dot V)^T$
are the same):
\begin{equation}\label{DElambdas}
\dot \lambda_i = (V^T\dot A V)_{ii} + 2 \lambda_i(V^TB\dot V)_{ii}  \ ,\,\, i=1,\dots, n\ ,
\end{equation}
that is the eigenvalues in general satisfy a linear non-homogeneous differential equation.

Next, differentiating \eqref{want}-(b), we obtain $\dot V^TBV+V^T\dot B V+V^TB\dot V=0 $ from which
we get 
\begin{equation}\label{Step2}
(V^TB \dot V)^T+(V^TB\dot V)= -(V^T\dot BV) \ ,
\end{equation}
hence we can obtain an expression for the symmetric part of $(V^TB\dot V)$, and in
particular in \eqref{DElambdas} we can use
\begin{equation*}%\label{DElamdbdas2}
2(V^TB\dot V)_{ii} \ = \  -(V^T\dot BV)_{ii}\ , \,\ i=1,\dots, n\ .
\end{equation*}
What we are missing is an expression for the anti-symmetric part of $(V^TB\dot V)$.  To arrive at this,
we use \eqref{Step1} relative to the off-diagonal entries.  This gives the following for the $(i,j)$ and $(j,i)$ entries:
\begin{equation}\label{Step3a}\begin{split}
& \lambda_i(V^T B \dot V)_{ij}+(\dot V^T B V)_{ij}\lambda_j = -(V^T \dot A V)_{ij}  \quad \text{or} \\
& (\lambda_i+\lambda_j)(V^T B \dot V)_{ij}+\lambda_j\left[ (\dot V^T B V)_{ij}-(V^T B \dot V)_{ij}\right] = -(V^T \dot A V)_{ij} \ 
\end{split}\end{equation}
and
\begin{equation}\label{Step3b}\begin{split}
& \lambda_j(V^T B \dot V)_{ji}+(\dot V^T B V)_{ji}\lambda_i = -(V^T\dot A V)_{ji} \quad \text{or} \\
& (\lambda_i+\lambda_j)(V^T B \dot V)_{ji}+\lambda_i\left[ (\dot V^T B V)_{ji}-(V^T B \dot V)_{ji}\right] = -
(V^T\dot A V)_{ji}  \quad \text{or} \\ 
& (\lambda_i+\lambda_j)(\dot V^T B V)_{ij}+\lambda_i\left[ (V^T B \dot V)_{ij}-(\dot V^T B V)_{ij} \right] =
-(V^T \dot AV)_{ij}\ , 
\end{split}\end{equation}
where we have used symmetry of $V^T\dot A V$ and the fact that the $(i,j)$-th entry of a matrix is the $(j,i)$-th entry
of its transpose.  Now, adding the last two expressions in \eqref{Step3a} and \eqref{Step3b}, we  obtain
\begin{equation}\label{Step4}\begin{split}
& (\lambda_i+\lambda_j)\left[(V^T B\dot V)_{ij}+ (\dot V^TB V)_{ij}\right]\ + \\
(\lambda_j-\lambda_i) & \left[ (\dot V^TBV)_{ij}-(V^TB\dot V)_{ij}\right] \ = \ -2(V^T\dot AV)_{ij}\ ,
\end{split}\end{equation}
and thus we can obtain an expression for the antisymmetric part of $(V^TB\dot V)$, upon using \eqref{Step2} for its
symmetric part.

So, finally, using \eqref{Step2} and \eqref{Step4}, we can obtain a formula for the term $V^TB\dot V$ which
depends on $B,\dot B, \Lambda$ and $V$. Let us formally set $C=V^TB\dot V$, and summarize 
the sought differential equations for $V$ and $\Lambda$:
\begin{equation}\label{DEforV}\begin{split}
\dot V \ & = \ VC\ , \,\, V(0)=V_0 \ , \\
C+C^T\ & = -V^T\dot B V\ ,   \itext{that is} C_{ij}+C_{ji}=-(V^T\dot B V)_{ij} \ , \\
\text{and} \quad C_{ij}-C_{ji}\ & = \
\frac{1}{\lambda_j-\lambda_i}\left[ 2(V^T\dot A V)_{ij}+(\lambda_i+\lambda_j)(C_{ij}+C_{ji})\right]\ , \\ 
\dot \Lambda  \ & = \ \diag(V^T\dot AV)- \Lambda \diag (V^T\dot B V)\ ,\,\, \Lambda(0)=\Lambda_0 \ , 
\end{split}\end{equation}

\begin{exm}[Standard Eigenproblem]
The most important special case of the previous analysis is of course the case where $B=I$, the standard
eigenproblem.  In this case, since $\dot B=0$, we obtain major simplifications.
For one thing, \eqref{DElambdas} is a simple integral not a linear differential equation for the eigenvalues:
\begin{equation}\label{DElambdasI}
\dot \lambda_i = (V^T\dot A V)_{ii}  \ ,\,\, i=1,\dots, n\ .
\end{equation}
Further, from \eqref{Step2} we observe that $V^T\dot V$ must be anti-symmetric, and thus we have that
$C=V^T\dot V$ is such that $C^T=-C$.  Hence, \eqref{Step4} simplifies to read
\begin{equation}\label{Step4I}
\dot V=VC\ ,\quad C_{ij}\ = \  \frac{(V^T\dot AV)_{ij}}{\lambda_j-\lambda_i}, \ i\ne j\ ,\,\, C_{ii}=0 \ 
\end{equation}
Formulas \eqref{DElambdasI} and \eqref{Step4I} of course match those derived for the
standard eigenproblem in \cite{DieciEirola}.
\end{exm}

\subsection{Periodicity}

To justify our algorithms to locate conical intersections, we will need being able to smoothly find
eigenvalues of the pencil (under the assumption that the eigenvalues
are distinct) along a closed loop in parameter space.  For this reason, we next give some results on
periodicity for the square root and the Cholesky factors of a positive definite periodic function, as
well as some general results on periodicity.

To begin with, let us properly define what we mean by a periodic function, and give an elementary
result on periodicity of the square root of a function.

\begin{defn}
A function $f\in\cont^k(\R,\R)$ ($k\ge 0$) is called periodic of period $1$, or simply $1$-periodic,
if $f(t+1)=f(t)$, for all $t$.  Moreover, we say that $1$ is the minimal period of $f$
if there is no $\tau<1$ for which $f(t+\tau)=f(t)$, for all $t$.  In the same way, we
say that the pencil $(A,B)$ is periodic of period $1$ if $A(t+1)=A(t)$ and $B(t+1)=B(t)$,
and further of minimal period $1$ if either $A$ or $B$ is such.
\end{defn}

\begin{lem}\label{PeriodicScalarRoot}
Let the real valued function $f\in\cont^k(\R,\R)$, $k\ge 0$, be strictly positive for all $t$, and
let $f$ be periodic of minimal period $1$.  
Let $s(t)=\sqrt{f(t)}$, $t\in \R$.  Then also $s$ is $\cont^k$ and periodic of minimal period $1$.
\end{lem}

\begin{proof} The smoothness result is in Lemma \ref{SqRootFtn}.  
For the periodicity, we argue by contradiction. \\
Observe that surely $s(t+1)=s(t)$, for all $t$, as otherwise one could not have $f(t+1)=f(t)$.
Then, if there is $\tau<1$ s.t. $s(t+\tau)=s(t)$,
for all $t$, then also $s(t)s(t)$ would be $\tau$-periodic that is $f$ would be $\tau$-periodic.
\end{proof}

Finally, we show that the Cholesky factor and the positive definite square root of a $1$-periodic
positive definite function are also $1$-periodic.

\begin{thm}\label{PeriodicFactors}
Let the function $A\in \cont^k(\R,\Rnxn)$, $k\ge 0$, be symmetric positive definite and of minimal period $1$.
\begin{itemize}
\item[(a)]
Let $L$ be the unique Cholesky factor of $A$: $A(t)=L(t)L^T(t)$, where $L$ is lower
triangular with positive diagonal, for all $t$.  Then, also $L$ has minimal period $1$.
\item[(b)]
Let $S=A^{1/2}$ tbe the unique positive definite square root of $A$: $S=S^T\succ 0$, 
$S^2=A$.  Then, also $S$ has minimal period $1$.
\end{itemize}
\end{thm}
\begin{proof}  First, consider the case of the Cholesky factor.
From Theorem \ref{SmoothChol}, we know that $L$ is as smooth as $A$, and $L(t)L^T(t)=A(t)$ for all $t$.
Now, since $A(t+1)=A(t)$, then $A(t+1)=L(t)L^T(t)$ as well as $A(t+1)=L(t+1)L^T(t+1)$.  From uniqueness
of the Cholesky factor, we then must have $L(t)=L(t+1)$.
%and the Cholesky factor is unique, we must have that $L$ satisfies $L(t+1)=L(t)$, for all $t$. 
Finally, if $L$ had minimal period $\tau<1$,
then necessarily so would $A$, but this contradicts that the minimal period of $A$ is $1$.

The proof for the square root is quite similar.  Using Theorem \ref{smoothness}, we know that $S$ is 
as smooth as $A$ and $S^2(t)=A(t)$ for all $t$.  Since $A(t+1)=A(t)$, then also $S(t+1)$ is a positive
definite square root of $A(t)$.  Since the square root is unique, we then have $S(t+1)=S(t)$.
As before, if $S$ had minimal period $\tau<1$,
then necessarily so would $A$, contradicting that the minimal period of $A$ is $1$.
\end{proof}

The next result is a corollary to Theorem \ref{BlockDiag} and will come in handy.

\begin{cor}\label{SimplePeriod1}
Let $V\in \cont^k(R,\Rnxn)$ be the function of which in Theorem \ref{BlockDiag}.
Let $\Gamma$ be a simple closed curve in $R$,
parametrized as a $\cont^p$ ($p\ge 0$) function $\gamma$ in the variable $t$,
so that the function $\gamma:\ t\in \R \to R$ is $\cont^p$ and of (minimal) period $1$.
Let $m=\min(k,p)$, and let $V_{\gamma}$ be the $\cont^m$ function $V(\gamma(t))$, $t\in \R$.
Then, $V_{\gamma}$ is $\cont^m$ and $1$-periodic.
\end{cor}
\begin{proof}
The result is immediate upon considering the composite function $V_{\gamma}$ and using
the stated smoothness and periodicity results.
\end{proof}

\begin{rem}\label{DistinctAndPeriodic}
In case the eigenvalues of the pencil $(A,B)$ are distinct in $R$, then the $B$-orthogonal function $V$
has diagonalized $(A,B)$.  For given ordering of the eigenvalues, as we already remarked $V$ is essentially unique:
the degree of non-uniqueness is given only by the signs of the columns of $V$.  Naturally, in this case 
Corollary \ref{SimplePeriod1} will give that a smooth $V_\gamma$ will be a 1-periodic function.
\end{rem}

The last result we give is a generalization of \cite[Lemma 1.7]{DiPu1} and it essentially states
that if the pencil $(A,B)$ has minimal period $1$, then there cannot coexist continuous eigendecompositions
of minimal periods $1$ and $2$.

\begin{lem}\label{1not2}
Let the functions $A=A^T\in \cont^k(\R,\Rnxn)$, $B=B^T\in \cont^k(\R,\Rnxn)\succ 0$, $k\ge 0$, 
be of minimal period $1$ and let the pencil $(A,B)$ have distinct eigenvalues for all $t$.
Suppose that %(the pencil $(A,B)$ is continuously diagonalizable)
there exists $V\in \cont^0$, invertible, and diagonal $\Lambda$ such that
$$A(t)V(t)=B(t)V(t)\Lambda(t)\ ,\,\, \forall t\ ,$$
with:
\begin{itemize}
\item[(i)] $\Lambda\in\mathcal{C}^0(\R,\Rnxn)$ diagonal with
distinct diagonal entries, and s.t. $\Lambda(t+1)=\Lambda(t)$; 
\item[(ii)] $V\in\mathcal{C}^0(\R,\Rnxn)$ invertible, with 
$$V(t+1)=V(t)\ D\ ,\,\, \forall t\in \R\ ,$$ where
$D$ is diagonal with $D_{ii}=\pm 1$ for all $i$, but $D\ne I_n$.
\end{itemize}
Then, there cannot exist an invertible continuous matrix function $T$ diagonalizing the pencil and of
period $1$.
% such that:
% \begin{itemize}
% \item[ ] $T \in \cont^0(\R,\Rnxn)$, invertible, of minimal period $1$,
% and eigendecomposing the pencil.
% % (that is, for which $A(t)T^{-1}(t)=B(t)T^{-1}(t)\Lambda(t)$ for all $t\in\R$.
% \end{itemize}
\end{lem}
\begin{proof}
By contradiction, suppose that there exists continuous $T$ of period $1$ such that
$A(t)T^{-1}(t)=B(t)T^{-1}(t)\Lambda(t)$, for all $t\in\R$. Therefore, we must have
$A=BT^{-1}\Lambda T$ and $A=BV\Lambda V^{-1}$ from which $\Lambda (TV)=(TV)\Lambda$.
But, $\Lambda(t)$ has distinct diagonal entries for all $t\in\R$, so that
$T(t)V(t)$ must be diagonal for all $t\in\R$.
Denote its diagonal entries by $c_1(t),\ldots,c_n(t)$, and so (since $TV$ is
invertible) $c_i\ne 0$, for all $t$.
But $T(t+1)V(t+1)=T(t)V(t)D$, for all $t\in\R$, hence there must exist an index
$i$ for which $c_i(t+1)=-c_i(t)$, which is a contradiction,
since the functions $c_i$'s are continuous and nonzero for $t\in\R$.
\end{proof}

\vspace{.5cm}

\section{Coalescing eigenvalues of \eqref{GenEv}}\label{CIs}

In this section, we study the occurrence of equal eigenvalues for \eqref{GenEv} when $A$ and $B$ depend
on two (real)  parameters.  We follow the skeleton of arguments given in \cite{DiPu1} for the symmetric
eigenproblem, and somewhat similar arguments to those used there.  Still, the extension to the symmetric
positive definite pencil is not automatic and needs to be done carefully.

First, we consider the case of a single pair of eigenvalues
coalescing, then generalize to several pairs coalescing at the same parameter values.
To begin with, we show that having a pair of coalescing eigenvalues is a codimension $2$ property.

\subsection{One generic coalescing in $\Omega$}

First, consider the $2\times 2$ case.  The following simple result is the key to
relate a {\emph{generic coalescing}} to the transversal intersection of two curves.

\begin{thm}\label{2x2Coal}
Let $A=A^T\in \cont^k(\Omega, \Rtbt)$ and $B=B^T\succ 0 \in \cont^k(\Omega, \Rtbt)$, $k\ge 1$.
Write $A(x)=\bmat{a & b \\ b & c}$ and $B(x)=\bmat{\alpha & \beta \\ \beta & \gamma}$.  Then, the
generalized eigenproblem 
\begin{equation}\label{eq:GenEvsProb_2x2}
\left(A-\lambda B\right)v=0 \ ,
\end{equation}
has identical eigenvalues at $x$ if and only if  
\begin{equation}\label{EqualEvaluesCod2}
\left\{\begin{array}{l} a\gamma \ = \ \alpha c \\ 
(a\gamma+c\alpha)\beta \ = \ 2\alpha \gamma b \end{array}\right. 
\iff \left\{\begin{array}{l} a\gamma \ = \ \alpha c \\ 
c\beta \ = \ \gamma b \end{array}\right. \ .
\end{equation}
\end{thm}
\begin{proof}
The problem 
$\left(\bmat{ a & b \\ b & c}-\lambda\bmat{ \alpha & \beta \\ \beta & \gamma }\right)v=0$ can be rewritten as
$\left(\bmat{\tilde a & \tilde b \\ \tilde b & \tilde c}-\lambda\bmat{ 1 & \tilde d \\ \tilde d &1}\right)w=0$,
where $\tilde a=a/\alpha$, $\tilde b=b/\sqrt{\alpha \gamma}$, $\tilde c=c/\gamma$ and
$\tilde d=\beta/\sqrt{\alpha \gamma}$, and $w=\bmat{v_1\sqrt{\alpha} \\ v_2\sqrt{\gamma}}$.
We observe that the sign of the entries of $v$ and of $w$ is the same, and we also note that $\tilde d^2<1$ (since
$B$ is positive definite).

We further have the following chain of equalities:
\begin{equation*}\begin{split}
& \left(\bmat{\tilde a & \tilde b \\ \tilde b & \tilde c}-\lambda\bmat{ 1 & \tilde d \\ \tilde d &1}\right)w=0  \ \iff  \\
& \left(\bmat{\tilde a & \tilde b \\ \tilde b & \tilde c}-\frac{\tilde c+\tilde a}{2}\bmat{ 1 & \tilde d \\ \tilde d &1}-
(\lambda-\frac{\tilde c+\tilde a}{2})\bmat{ 1 & \tilde d \\ \tilde d &1}\right)w=0 \\
& \iff  \left(\bmat{\hat a & \hat b \\ \hat b & -\hat a} -  \mu \bmat{ 1 & \tilde d \\ \tilde d &1}\right)w=0 \ ,
\end{split}\end{equation*}
where 
\begin{equation}\label{hats}
\hat a=\frac{\tilde a-\tilde c}{2}\ , \,\ \hat b=\tilde b-\frac{\tilde a+\tilde c}{2}\tilde d\ , \,\
\mu=\lambda-\frac{\tilde a+\tilde c}{2}\ .
\end{equation}
(Note that we have reduced the problem to one for which $A$ has $0$-trace.)
Now, an explicit computation gives
\begin{equation}\label{mus}
\mu_{1,2}(t)=\frac{-\hat b\tilde d\pm 
\sqrt{\hat b^2+(1-\tilde d^2)\hat a^2 }}{1-\tilde d^2}\,.
\end{equation}
Now, we have identical eigenvalues $\mu$ (hence $\lambda$) if and only if 
% the discriminant $\Delta$ of the characteristic polynomial is zero.  For the discriminant we have
% \begin{equation*}
% %\Delta=(1-\tilde d^2)(\tilde a-\tilde c)^2+(2\tilde b-\tilde d(\tilde a+\tilde c))^2\ ,
% \Delta=(1-\tilde d^2)(\hat a^2+\hat b^2)+(\hat b\tilde d)^2\ ,
% \end{equation*}
% and therefore 
% $$\Delta=0\ \text{ if and only if }\ \left\{\begin{array}{l} \hat a = 0 \\ \hat b = 0 \end{array}\right. .$$
$$ \left\{\begin{array}{l} \hat a = 0 \\ \hat b = 0 \end{array}\right. .$$
Rephrasing in terms of the original entries, this is precisely what we wanted to verify.
\end{proof}

We now have

\begin{thm}[$2\times 2$ case] \label{2by2}
Consider $A=A^T \in \cont^k(\Omega,\Rtbt)$, 
and $B=B^T\succ 0 \in \cont^k(\Omega, \Rtbt)$, $k\ge 1$.
For all $x\in \Omega$, write
$$A(x)=\bmat{a & b \\ b & c}\ ,\quad B(x)=\bmat{\alpha & \beta \\ \beta & \gamma}$$
and let $\lambda_1$ and $\lambda_2$ be the two continuous eigenvalues of the pencil $(A,B)$,
and labeled so that $\lambda_1(x) \ge \lambda_2(x)$ for all $x$ in
$\Omega$. Assume that there exists a unique point $\xi_0\in \Omega$
where the eigenvalues coincide: $\lambda_1(\xi_0)=\lambda_2(\xi_0)$.
Consider the $\cont^k$ function $F:\ \Omega \to \R^2$ given by
\begin{equation*}%\label{FSymm}
F(x)=\bmat{a(x)\gamma(x)-\alpha(x)c(x) \\ b(x)\gamma(x)-\beta(x)c(x)  }\,,
\end{equation*}
and assume that $0$ is a regular value for both function $a\gamma-\alpha c$ and $b\gamma -\beta x$.
\footnote{This implies that the zeros set of these functions is actually a $\cont^k$ curve (or collection
of $\cont^k$ curves).  For background on these concepts, see \cite{Hirsch:DiffTop}}
Then, consider the two $\cont^k$ curves $\Gamma_1$ and $\Gamma_2$ through $\xi_0$,
given by the zero-set of the components of $F$:
$\Gamma_1=\{x\in \Omega: \, a(x)\gamma(x)-\alpha(x)c(x) =0\}$, 
$\Gamma_2=\{x\in \Omega: \,  b(x)\gamma(x)-\beta(x)c(x)=0\}$.
Assume that $\Gamma_1$ and $\Gamma_2$ intersect transversally at $\xi_0$.
\footnote{Transversal intersection means that the two tangents to the curves at $\xi_0$
are not parallel to each other}

Let $\Gamma$ be a simple closed curve enclosing the point $\xi_0$, and let it
be parametrized as a $\cont^p$ ($p\ge 0$) function $\gamma$ in the variable $t$,
so that the function $\gamma:\ t\in \R \to \Omega$ is $\cont^p$ and $1$-periodic.
Let $m=\min(k,p)$, and let $A_{\gamma}$, $B_{\gamma}$ be the $\cont^m$ functions 
$A(\gamma(t))$, $B(\gamma(t))$, for all $t\in \R$.
Then, for all $t\in \R$, the pencil 
$(A_{\gamma}, B_{\gamma})$ has the eigendecomposition
$$A_{\gamma}(t)V(t)=B_{\gamma}(t)V_{\gamma}(t)\Lambda_{\gamma}(t)$$
such that:
\begin{itemize}
\item[(i)] $\Lambda_{\gamma}\in\mathcal{C}^m(\R,\Rtbt)$ and diagonal:
$\Lambda_{\gamma}(t)=\bmat{\lambda_1(\gamma(t)) & 0 \\ 0 & \lambda_2(\gamma(t))}$, and
$\Lambda_{\gamma}(t+1)=\Lambda_{\gamma}(t)$ for all $t\in \R$;
\item[(ii)] $V_{\gamma}\in\mathcal{C}^m(\R,\Rtbt)$, $V_{\gamma}(t+1)=-V_{\gamma}(t)$
for all $t\in \R$, and $V_\gamma$ is $B_\gamma$-orthogonal: 
$V_\gamma(t)^TB_\gamma(t)V_{\gamma}(t)=I$, for all $t\in \R$.
\end{itemize}
\end{thm}

\begin{proof} The proof follows closely the one used in \cite[Theorem 2.2]{DiPu1} for
the symmetric eigenproblem, with the necessary changes due to the dealing with the
generalized eigenproblem, and also fixing some imprecisions in the proof of \cite[Theorem 2.2]{DiPu1}.

Because of Theorem \ref{2x2Coal},
$$\lambda_1(x)=\lambda_2(x)\  \Longleftrightarrow \ F(x)=\bmat{0 \\ 0}\ ,$$
and, by hypothesis, $\xi_0$ is the unique root of $F(x)$ in $\Omega$.  Moreover, 
under the assumption of $\xi_0$ being the only root of $F$
in $\Omega$, just like in the proof of Theorem \ref{2x2Coal}, see \eqref{hats}, 
we can also rewrite the problem in the simpler form
\begin{equation*}%\label{SimplerSymm}
F(x)=0 \iff G(x)=0\itext{where} G(x)=\bmat{\hat a(x) \\ \hat b(x) }\,.
\end{equation*}
Further, $0$ is a regular value for both function $\hat a$ and $\hat b$, and therefore $G(x)=0$
continues to define smooth curves intersecting transversally at $\xi_0$, call them $\hat \Gamma_1$
and $\hat \Gamma_2$ (these are just rescaling and shifting of the curves $\Gamma_1$ and $\Gamma_2$).  
Moreover, we let $(\hat A, \hat B)$ be the pencil associated to these simpler functions:
$$\hat A(x)=\bmat{\hat a(x) & \hat b(x) \\ \hat b(x) & -\hat a(x)}\ , \quad 
\hat B(b)=\bmat{1 & \tilde d(x) \\ \tilde d(x) & 1}\ ,\,\, x\in \Omega\ .$$

At this point, we will prove the asserted results for $\hat \Gamma_1$ and $\hat \Gamma_2$ by
first showing that it holds true along a small circle $C$ around $\xi_0$, and then
show that the same results hold when we continuously deform $C$ into $\Gamma$.

Since $\hat \Gamma_1$ and $\hat \Gamma_2$ intersect transversally at $\xi_0$, we let
$C$ be a circle centered at $\xi_0$, of
radius small enough so that the circle goes through each of
$\hat \Gamma_1$ and $\hat \Gamma_2$ at exactly two distinct points, see Figure \ref{figure1}.

\begin{figure}
  \center
%  \psfrag{a-d}{$\hat a(x)=0$}
%  \psfrag{b}{$\hat b(x)=0$}
%  \psfrag{gamma}{$\Gamma$}
%  \psfrag{Omega}{$\Omega$}
%  \psfrag{xi}{$\xi_0$}
%  \psfrag{C}{$C$}
  \includegraphics[width=.75\textwidth]{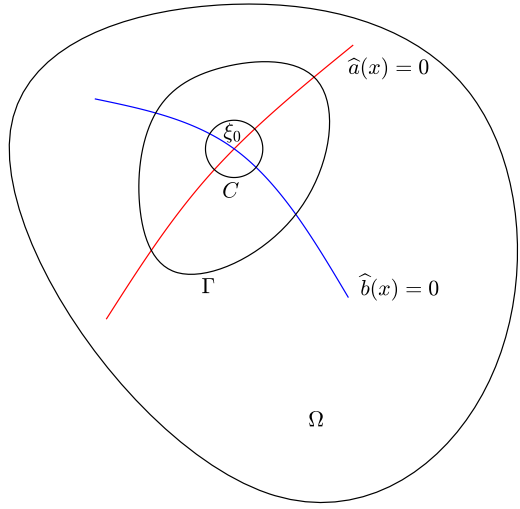}
  \caption{Transversal Intersection at $\xi_0$}
  \label{figure1}
\end{figure}

Further, let $C$ be parametrized by a continuous $1$-periodic function
$\rho$, $\rho(t+1)=\rho(t)$, for all $t\in\R$.

Consider the pencil $(\hat A(\rho(t)), \hat B(\rho(t)))$, $t\in \R$, which is thus
a smooth (and $1$-periodic) pencil, with distinct eigenvalues, so that its smooth
eigenvalues $\mu_{1,2}$ in \eqref{mus} (where
all functions $\hat a, \hat b, \tilde d$ are evaluated along $C$)
%(these are shifted values of the eigenvalues of the pencil $(A,B)$ along
%the circle $C$) 
will necessarily satisfy $\mu_j(t+1)=\mu_j(t)$, $j=1,2$.
The smooth eigenvectors of $(\hat A(\rho(t)), \hat B(\rho(t)))$, call them $W_\rho(t)$,
are uniquely determined (for each $t$) up to sign.  Call $\bmat{u_1\\ u_2}$ the eigenvector 
relative to $\mu_2$, so that 
$$\left[ 
\begin{pmatrix} \hat a & \hat b \\ \hat b & -\hat a \end{pmatrix} -\mu_2
\begin{pmatrix} 1 & \tilde d \\ \tilde d & 1 \end{pmatrix} \right]
\begin{pmatrix} u_1 \\ u_2 \end{pmatrix}=0\ .$$
From this, a direct computation shows that (recall that, presently, all functions
are computed along $C$)
$$\begin{cases}
\left(\hat a(1-\tilde d^2)+\hat b \tilde d+\sqrt{\hat b^2+(1-\tilde d^2)\hat a^2 }\right) u_1 &=
-\left(\hat b+\tilde d\sqrt{\hat b^2+(1-\tilde d^2)\hat a^2 }\right)u_2 \\
\left(-\hat a(1-\tilde d^2)+\hat b \tilde d+\sqrt{\hat b^2+(1-\tilde d^2)\hat a^2 }\right) u_2 &=
-\left(\hat b+\tilde d\sqrt{\hat b^2+(1-\tilde d^2)\hat a^2 }\right)u_1 \,.
\end{cases}$$
Therefore, from these it follows that $u_1$ (respectively, $u_2$)
changes sign if and only if $\hat b$ goes through zero and $\hat a>0$
(respectively, $\hat a < 0$). Therefore, each of the two
functions $u_1$ and $u_2$ changes sign only once over any interval of length $1$, and 
since no continuous function of period $1$ can change sign only once over one period, 
it follows that $u_1$ and $u_2$ must be $2$-periodic functions and the periodicity assertions
of the theorem follow relatively to the curve $\rho(t)$ for the eigenvector function $W$.  That is,
along $C$ we have that $W$ has period $2$.
Finally, we note that the eigenvector function $V$ has columns whose entries have the same sign
as those of $W$ (see the third line in the proof og Theorem \ref{2x2Coal}), 
so that the periodicity assertion holds for $V$.

Finally, the extension from the circle $C$ to the curve $\Gamma$ enclosing the point $\xi_0$ follows 
in the same way as was done in \cite{DiPu1}; in particular, see the final part of the
proof of Theorem 2.2 and Remark 2.5 in there. 
\end{proof}

The assumption of transversality for the curves $\Gamma_1$
and $\Gamma_2$ at $\xi_0$ is generic within the class of smooth curves intersecting at a point.
As a consequence, we can say 
that $\xi_0$ is a {\em generic coalescing point of eigenvalues} of \eqref{eq:GenEvsProb_2x2}
when $\Gamma_1$ and $\Gamma_2$ intersect transversally at $\xi_0$.
As a consequence, within the class of $\cont^k$ functions $A,B$, generically we will need
two parameters to observe coalescing of the eigenvalues of \eqref{eq:GenEvsProb_2x2},
and such coalescings will occur at isolated points in parameter space and persist
(as a phenomenon, the parameter value will typically change) under generic perturbation.

\begin{exm}\label{ExampleCI}
Take $A(x,y)= \bmat{4x+3y & 5y \\ 5y & -4x+3y}$, $B(x,y)=\bmat{5 & 3 \\ 3 & 5}$.  Then, the eigenvalues satisfy
the relation $\lambda_{1,2}=\pm \sqrt{x^2+y^2}$,
\eqref{EqualEvaluesCod2} gives the solution $x=y=0$, and the eigenvalues are not differentiable there.
If we perturb the data as $A\to A+\epsilon \bmat{1 & 1 \\ 1 & -1}$,
then the solution of \eqref{EqualEvaluesCod2} is $x=-\epsilon/4$, $y=-5\epsilon/16$. \qed
\end{exm}

Using Theorem \ref{2by2}, and Theorem \ref{BlockDiag}, 
we can characterize the case of a symmetric-definite pencil in $\Rnxn$, whose eigenvalues coalesce
at a unique point $\xi_0$.  

\begin{defn}\label{TransCoalEVs}
Let $A=A^T\in\cont^k(\Omega,\Rnxn)$, $B=B^T\succ 0\in \cont^k(\Omega,\Rnxn)$, and let
$\lambda_1(x),\ldots,\lambda_n(x)$, $x\in \Omega$, be the continuous eigenvalues of the
pencil $(A,B)$, ordered so that 
$$\lambda_1(x)>\lambda_2(x)>\ldots>\lambda_k(x)\ge\lambda_{k+1}(x)>
\ldots>\lambda_n(x)\ ,\, \forall x\in \Omega\ ,$$
and
$$\lambda_k(x)=\lambda_{k+1}(x)\Longleftrightarrow x=\xi_0\in \Omega\ .$$
Let $R$ be a rectangular region $R\subseteq \Omega$ containing $\xi_0$ in its interior.
Moreover, let
\begin{itemize}
\item[(1)] $V\in \cont^k(R,\Rnxn)$ be a $B$-orthogonal function achieving
the reduction guaranteed by Theorem \ref{BlockDiag}:
\begin{equation*}\begin{split}
& V^T(x)A(x)V(x)=\bmat{\Lambda_1(x) & 0 & 0 \\ 0 & \tilde A(x) & 0 \\ 0 & 0 & \Lambda_2(x)}\ ,\,\, \text{and} \\
& V^T(x)B(x)V(x)=\bmat{I_{k-1} & 0 & 0 \\ 0 & \tilde B(x) & 0 \\ 0 & 0 & I_{n-k-1}}\ ,\,\ 
\forall x\in R\ , 
\end{split}\end{equation*}
where
$\Lambda_1\in\cont^k(R,\R^{(k-1) \times (k-1)})$ and
$\Lambda_2\in\cont^k(R,\R^{(n-k-1) \times (n-k-1)})$, such that, for all $x\in R$,
$\Lambda_1(x)=\diag(\lambda_1(x),\ldots,\lambda_{k-1}(x))$, and
$\Lambda_2(x)=\diag(\lambda_{k+2}(x),\ldots,\lambda_{n}(x))$.  Moreover,
$\tilde A=\tilde A^T\in\cont^k(R,\Rtbt)$, $\tilde B=\tilde A^T\succ 0 \in\cont^k(R,\Rtbt)$
and the pencil $(\tilde A, \tilde B)$ has 
eigenvalues $\lambda_k(x),\lambda_{k+1}(x)$ for each $x\in R$;
\item[(2)]  for all $x\in R$, write $\tilde A(x)=\bmat{a(x) & b(x) \\ b(x) & d(x)}$, 
$\tilde B(x)=\bmat{\alpha(x) & \beta(x) \\ \beta(x) & \gamma(x)}$.  Assume that $0$ is
a regular value for the functions $a\gamma-\alpha c$ and $b\gamma -\beta c$, 
and define the function $F$ and the curves $\Gamma_1$ and $\Gamma_2$
as in Theorem \ref{2by2}.
\end{itemize}
Then, we call $\xi_0$ a {\em generic coalescing point of eigenvalues} in $\Omega$,
if the curves $\Gamma_1$ and $\Gamma_2$ intersect transversally at $\xi_0$.
\end{defn}

\begin{rem}
Arguing in a similar way to \cite[Theorem 2.7]{DiPu1}, it is a (lengthy, but simple)
computation to verify that Definition \ref{TransCoalEVs} is independent of the transformation
$V$ used to bring the pencil $(A,B)$ to block-diagonal form.
\end{rem}

\begin{cor}\label{cod2}
Having exactly a pair of equal eigenvalues of \eqref{GenEv} is a codimension 2 phenomenon.
\end{cor}
\begin{proof}
This is because coalescence is expressed by the two relations in \eqref{EqualEvaluesCod2}, or --as seen
in the proof of Theorem \ref{2x2Coal}-- by the two relations
$\left\{\begin{array}{l} \hat a = 0 \\ \hat b = 0 \end{array}\right.$.  This, coupled with
Definition \ref{TransCoalEVs}, gives the claim.
\end{proof}

As a consequence of its definition, and of Corollary \ref{cod2},
for a coalescing point of eigenvalues of a two-parameter symmetric-definite
pencil to be a generic coalescing point is a generic property.

\begin{rem}
Although the above reasoning on the codimension is done relative to matrices $A$ and $B$ that
are ``full'', the stated codimension does not change when $A$ and $B$ are banded functions, both
with bandwidth $b\ge 1$.  For example, this fact can be appreciated by pointing out that 
the pencil $(A-\lambda B)v=0$ has same eigenvalues as the symmetric eigenproblem 
$(\tilde A-\lambda I)w=0$; e.g., with $w=L^Tv$ and $B=LL^T$.  Although $L$ is banded when
$B$ is so, the function $\tilde A$ is full, hence the codimension of having a pair of
equal eigenvalues is the same as that of a symmetric eigenproblem having a pair of
equal eigenvalues, which is $2$.
\end{rem}

As already exemplified by Example \ref{ExampleCI}, at a point where eigenvalues of the
pencil coalesce, there is a complete loss of smoothness of the eigenvalues.  In fact,
the situation of Example \ref{ExampleCI} is fully general, as the next example shows.

\begin{exm}\label{ExGenCI}
Without loss of generality (see the proof of Theorem \ref{2x2Coal}), take the symmetric-definite
pencil $(A,B)$ with 
$$A(x)= \bmat{a(x) & b(x) \\ b(x) & -a(x)}\ , \,\ B(x)=\bmat{1 & d(x) \\ d(x) & 1}\,.$$
and let $\xi_0$ be such that $a(\xi_0)=b(\xi_0)=0$, and (because of transversality) we also
have that the Jacobian 
$\bmat{\nabla a \\ \nabla b}_{\xi_0}=
\bmat{a_x & a_y \\ b_x & b_y}_{\xi_0}$
is invertible, that is $a_xb_y-a_yb_x\ne 0$.
Now, the eigenvalues $\mu_{1,2}$ of the pencil are given by \eqref{mus}:
$\mu_{1,2}(x)=\frac{-b d\pm \sqrt{h(x)}}{1-d^2}$, with
$h(x)=b^2+(1-d^2) a^2$.   Now, expand the function $h(x)$ at $\xi_0$.  We get
$$h(x)=h(\xi_0)+\nabla h(\xi_0) (x-\xi_0)+\frac{1}{2}(x-\xi_0)^TH(\xi_0)(x-\xi_0)+\dots \,,$$
and a simple computation gives $h(\xi_0)=0$, $\nabla h(\xi_0)=0$, and
$$H(\xi_0)=2\bmat{b_x^2+a_x^2(1-d^2) & b_xb_y-a_xa_y(1-d^2) \\
b_xb_y-a_xa_y(1-d^2)  & b_y^2+a_y^2(1-d^2) }$$
so that at $\xi_0$: $H_{11}>0$, $H_{22}>0$, and $\det(H(\xi_0))=(1-d^2)(b_xa_y-a_xb_y)^2$
and this is positive, because of the previously remarked transversality.  Therefore,
$H(\xi_0)$ is positive definite,
and in the vicinity of $\xi_0$ the eigenvalues have the
form $\mu_{1,2} =
\frac{-b d\pm \sqrt{\|z\|^2+O(\|x-\xi_0\|^4)}}{1-d^2}$
where $z=H^{1/2}(\xi_0)\ (x-\xi_0)$.  As a consequence, the eigenvalues's surface have
a double cone structure at the coalescing point. This justifies calling the coalescing
point a {\emph{conical intersection}}, or CI for short.
\end{exm}

Obviously, there is a total loss of differentiability through a CI point.  Recall that
the applications motivating our study is dimension reduction through a projection approach;
but then CIs are particularly
bothersome since the projection looses uniqueness at a CI point.  For this reason, in this work
we emphasize detecting parameter values where CIs occur, in particular we give criteria that
enable detection of generic CIs.  The case of a $(2,2)$ pencil was dealt with in Theorem \ref{2x2Coal}.
The case of a $(n,n)$ pencil, with only a single generic coalescing of eigenvalues in $\Omega$ 
is dealt with in the next theorem.

\begin{thm} \label{nxncase}
Let $A=A^T\in\cont^k(\Omega,\Rnxn)$, 
$B=B^T\in\cont^k(\Omega,\Rnxn)\succ 0$, and let 
$\lambda_1(x),\ldots,\lambda_n(x)$, $x\in \Omega$, be the continuous
eigenvalues of the pencil $(A,B)$.  Assume that
$$\lambda_1(x)>\lambda_2(x)>\ldots>\lambda_j(x)\ge\lambda_{j+1}(x)>
\ldots>\lambda_n(x)\ ,\, \forall x\in \Omega\ ,$$
and
$$\lambda_j(x)=\lambda_{j+1}(x)\Longleftrightarrow x=\xi_0\in \Omega\ ,$$
where $\xi_0$ is a generic coalescing point.

Let $\Gamma$ be a simple closed curve in $\Omega$
enclosing the point $\xi_0$, and let it
be parametrized as a $\cont^p$ ($p\ge 0$) function $\gamma$ in the variable $t$,
so that the function $\gamma:\ t\in \R \to \Omega$ is $\cont^p$ and $1$-periodic.
Let $m=\min(k,p)$, and let $A_{\gamma}$, $B_{\gamma}$ be the $\cont^m$ 
restrictions of $A, B$, to $\gamma(t)$, $t\in \R$.

Then, for all $t\in \R$, the pencil $(A_{\gamma}, B_{\gamma})$ admits the diagonalization
$A_{\gamma}(t)V_\gamma(t)=B_{\gamma}(t)V_\gamma(t)\Lambda(t)$, where
\begin{itemize}
\item[(i)] $\Lambda\in\mathcal{C}^m(\R,\Rnxn)$, $\Lambda(t+1)=\Lambda(t)$, and
$\Lambda(t)=\diag(\lambda_1(t),\ldots,\lambda_{n}(t))$,
$\forall t\in \R$;
\item[(ii)] $V_\gamma\in\mathcal{C}^m(\R,\Rnxn)$ is $B$-orthogonal, and
$$V_\gamma(t+1)=V_\gamma(t) D\ ,\,\
D=\bmat{I_{j-1} & 0 & 0\\ 0 & -I_2 & 0\\
0 & 0 & I_{n-j-1}}\ .$$
\end{itemize}
\end{thm}
\begin{proof}
The proof combines the block-diagonalization result Theorem \ref{BlockDiag} with the $(2,2)$ case.

So, we consider a rectangle $R\subseteq \Omega$ around $\xi_0$, and consider a $B$-orthogonal function
$V\in \cont^e(R,\Rnxn)$ giving the block decomposition of Definition \ref{TransCoalEVs}
\begin{equation*}\begin{split}
& V^T(x)A(x)V(x)=\bmat{\Lambda_1(x) & 0 & 0 \\ 0 & \tilde A(x) & 0 \\ 0 & 0 & \Lambda_2(x)}\ ,\,\, \text{and} \\
& V^T(x)B(x)V(x)=\bmat{I_{j-1} & 0 & 0 \\ 0 & \tilde B(x) & 0 \\ 0 & 0 & I_{n-j-1}}\ ,\,\ 
\forall x\in R\ .
\end{split}\end{equation*}
Let $C$ be a circle enclosing $\xi_0$
and contained in $R$, parametrized by a continuous $1$-periodic
function $\rho$, and let $\tilde A_\rho(t)=\tilde A(\rho(t))$, $\tilde B_\rho(t)=\tilde B(\rho(t))$, 
$t\in \R$.
%  see Figure \ref{figure2}.
% 
% 
% \begin{figure}
%   \center
%   \psfrag{gamma}{$\Gamma$}
%   \psfrag{R}{$R$}
%   \psfrag{x0}{$\xi_0$}
%   \psfrag{C}{$C$}
%   \psfrag{b}{$b(x)=0$}
%   \psfrag{a-d}{$a(x)-d(x)=0$}
%   \includegraphics[width=3.25in]{figure2.eps}
%   \caption{Generic coalescing at $\xi_0$. }
%   \label{figure2}
% \end{figure}
% 
Let $V_\rho$ be the orthogonal function of Theorem \ref{2x2Coal}
associated to the pencil $(\tilde A_\rho, \tilde B_\rho)$, so that 
$V_\rho(t+1)=-V_\rho(t)$, for all $t$, and moreover $V_\rho^TB_\rho V_\rho=I_2$.
Now, consider the following continuous function
$$V(\rho(t))\bmat{I_{j-1} & 0 & 0\\ 0 & V_\rho(t) & 0\\
0 & 0 & I_{n-j-1}}\,,$$
Since $V(\rho(t+1))=V(\rho(t))$ for all $t$,
then the result follows relative to the circle $C$.
% $$U_\rho(t+1)\ =\ U_\rho(t) \bmat{I_{k-1} & 0 & 0\\ 0 & -I_2 & 0\\
% 0 & 0 & I_{n-k-1}}\ .$$
The argument that the same periodicity properties hold relative to
the simple closed curve $\Gamma$ follow similarly to what we did in the proof
of \cite[Theorem 2.8]{DiPu1}.
% .
% We do this in the same way as what we did in Theorem \ref{Symm2by2}.
% Take a homotopy $h(s,t)$, $(s,t)\in[0,1]\times[0,1]$, as in Theorem \ref{Symm2by2},
% and consider the function $A(h(s,t))$, $(s,t)\in[0,1]\times[0,1]$.
% $A(h(s,t))$ is continuous with distinct eigenvalues for all $(s,t)\in[0,1]\times[0,1]$,
% and so --by Theorem \ref{BlockDiag}-- we can write $A(h(s,t))=V(s,t)\Lambda(s,t)V^T(s,t)$, where
% $\Lambda(s,t)$ and $V(s,t)$ are continuous, $\Lambda(s,t)$ is diagonal, and
% $V(s,t)$ is real orthogonal.  Partition $V$ by columns:
% $V(s,t)=\bmat{v_1(s,t) & \cdots & v_n(s,t)}$.
% Let $f_j(s)=v_j^T(s,0)\,v_j(s,1)$, for $j=1,\dots, n$.
% Since $h(s,0)=h(s,1)$ for all $s\in[0,1]$,
% we have that all $f_j$'s take values in $\{-1, 1\}$.
% Being continuous, we must have
% $f_j(0)=f_j(1)=1$ for all $j\ne k,k+1$, and $f_j(0)=f_j(1)=-1$ for $j=k,k+1$,
% from which the result follows.
\end{proof}

% Theorem \ref{nxncase} has an interesting geometric interpretation: the
% $\cont^m$ eigenvectors of $A_\gamma$ corresponding to the
% eigenvalues coalescing at $\xi_0$ (i.e., the $k$-th and $(k+1)$-st eigenvectors),
% ``get upside down'' as we complete one loop along the closed curve $\Gamma$.
It is worth emphasizing that for the eigenvectors associated
to eigenvalues which do not coalesce inside $\Omega$, we have
$v_\gamma(t+1)=v_\gamma(t)$.  In other words, 
a continuous eigendecomposition $V$ along a simple curve $\Gamma$
not containing coalescing points inside (or on) it, satisfies $V(t+1)=V(t)$.
This consideration, coupled with the uniqueness up to sign of a $B$-orthogonal
function eigendecomposing a pencil with distinct eigenvalues, gives the following.

\begin{cor}\label{NoCoalPeriod1}
Let $(A,B)$ be a $\cont^k$ symmetric-positive definite pencil for all $x\in \Omega$.
Let $\Gamma$ be a simple closed curve in $\Omega$,
parametrized by the
$\cont^p$ and $1$-periodic function $\gamma$.  Let $m=\min(k,p)$, and  
let $(A_\gamma, B_\gamma)$ be the smooth pencil restricted to $\Gamma$.
If there are no coalescing points inside $\Gamma$ (nor on it), 
then any $\cont^m$ eigendecomposition $V$ of the pencil $(A_\gamma, B_\gamma)$
satisfies $V(t+1)=V(t)$.
\end{cor}

\subsection{Several generic coalescing points in $\Omega$}

Here we consider the case when several eigenvalues of the pencil
coalesce inside a closed curve $\Gamma$. In line with our previous analysis of generic cases,
we only consider the case when coalescing points are {\emph{isolated}} and generic, as
characterized next.

\begin{defn}\label{GenCoal}
Consider the pencil $(A,B)$, with $A=A^T\in \cont^k(\Omega, \Rnxn)$ and
$B=B^T\in \cont^k(\Omega, \Rnxn)\succ 0$, $k\ge 1$. 
A parameter value $\xi_0\in \Omega$ is called a 
{\em generic coalescing point of eigenvalues} if there is a pair of equal eigenvalues
at $\xi_0$, no other pair of eigenvalues coalesce
inside an open simply connected region
$\Omega_0\subseteq \Omega$, 
and $\xi_0$ is a generic coalescing point of eigenvalues in $\Omega_0$.
\end{defn}

In these cases, we have the following result.

\begin{thm} \label{PencilAny}
Consider the pencil $(A,B)$, where
$A=A^T\in\cont^k(\Omega,\Rnxn)$, 
$B=B^T\in\cont^k(\Omega,\Rnxn)\succ 0$, and let 
$\lambda_1(x)\ge\ldots\ge\lambda_n(x)$ be its continuous eigenvalues.
Assume that
for every $i=1,\ldots,n-1$,
$$\lambda_{i}(x)=\lambda_{i+1}(x)$$ at $d_i$ distinct
generic coalescing points in $\Omega$, so that there are
$\sum_{i=1}^{n-1}d_i$ such points\footnote{Of course, some $d_i$'s may be $0$}.
Let $\Gamma$ be a simple closed curve in $\Omega$ enclosing all of these
distinct generic coalescing points of eigenvalues, and let it be
parametrized as a $\cont^p$ ($p\ge 0$) function $\gamma$ in the variable $t$,
so that the function $\gamma:\ t\in \R \to \Omega$ is $\cont^p$ and $1$-periodic.
Let $m=\min(k,p)$ and let $A_\gamma$ and $B_\gamma$ be the $\cont^m$ restrictions of
$A$ and $B$ to $\gamma(t)$.    Then, for all $t\in \R$, there exists
$V$ diagonalizing the pencil $(A_\gamma, B_\gamma)$: $A_{\gamma}V=B_{\gamma}V\Lambda$,
where
\begin{itemize}
\item[(i)] $\Lambda\in\mathcal{C}^m(\R,\Rnxn)$ is diagonal:
$\Lambda=\diag\bigl(\lambda_1(t), \dots, \lambda_n(t)\bigr)$,
for all $t\in \R$, and $\Lambda(t+1)=\Lambda(t)$;
\item[(ii)] $V\in\mathcal{C}^m(\R,\Rnxn)$ is 
$B_\gamma$-orthogonal, with 
$$V(t+1)=V(t)\ D\ ,\,\, \forall t\in \R\ ,$$
where $D$ is a diagonal matrix of $\pm 1$ given as follows:
$$D_{11}=(-1)^{d_1},\ \
D_{ii}=(-1)^{d_{i-1}+d_i}\ \mathrm{for}\ i=2,\ldots,n-1,\ \
D_{nn}=(-1)^{d_{n-1}}\ .$$
In particular, if $D=I$, then $V$ is
1-periodic, otherwise it is 2-periodic with minimal period $2$.
\end{itemize}
\end{thm}

\begin{proof}
Since the eigenvalues are distinct on $\Gamma$, we know that there is a $\cont^m$ eigendecomposition
$V$ of the pencil $(A_\gamma, B_\gamma)$, and that $V$ is $B_\gamma$-orthogonal.
The issue is to establish the periodicity of $V$.  Our proof is 
by induction on the number of coalescing points.  

Because of Theorem \ref{nxncase}, we know that the result is true for $1$ coalescing point.
So, we assume that the result holds
for $N-1$ distinct generic coalescing points, and we'll show it for $N$ distinct generic
coalescing points; note that $N=\sum_{i=1}^{n-1}d_i$.

Since the coalescing points are distinct, we can always separate
one of them, call it $\xi_N$, from the other $N-1$ points,
with a curve $\alpha$ not containing coalescing
points, and which stays inside the region bounded by $\Gamma$,
joining two distinct points on $\Gamma$,
$y_0=\gamma(t_0)$ and 
$y_1=\gamma(t_1)$, with $t_0, t_1\in[0,1)$, so that $\alpha$ leaves
$\xi_{N}$ and all other coalescing points $\xi_{i}$'s on opposite sides
(see Figure \ref{figure2}). Let $j,\  1\le j\le n-1$, be the index for which
$\lambda_{j}(\xi_N)=\lambda_{j+1}(\xi_N)$.

\begin{figure}
  \center
%  \psfrag{gamma1}{$\Gamma_0$}
%  \psfrag{gamma2}{$\Gamma_1$}
%  \psfrag{x0}{$\cdot\ \xi_0$}
%  \psfrag{x1}{$\cdot\ \xi_1$}
%  \psfrag{x2}{$\cdot\ \xi_2$}
%  \psfrag{xn-1}{$\cdot\ \xi_{N-1}$}
%  \psfrag{xn}{$\cdot\ \xi_N$}
%  \psfrag{y0}{$y_0$}
%  \psfrag{y1}{$y_1$}
%  \psfrag{dots}{$\ldots$}
%  \psfrag{alpha}{$\alpha$}
  \includegraphics[width=.75\textwidth]{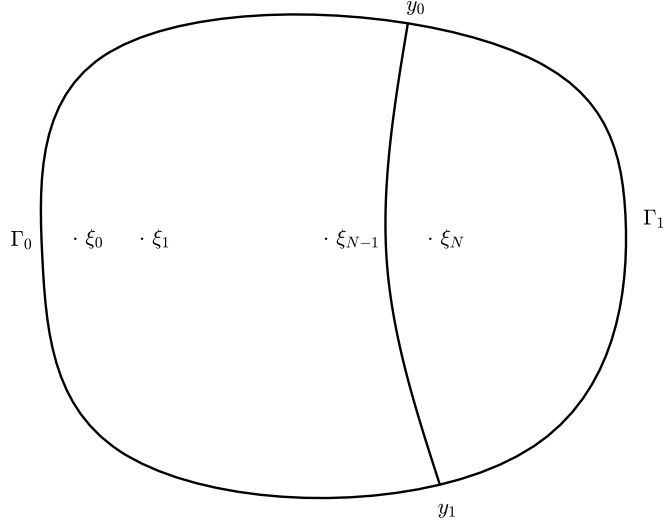}
  \caption{Figure for proof of Theorem
  \ref{PencilAny}}
  \label{figure2}
\end{figure}

Now consider the following construction.  
Take a smooth eigendecomposition of $(A_\gamma,B_\gamma)$ along
$\Gamma$, starting at $y_0$ and returning to it; the loop is done
once, and to fix ideas, we will transverse it in the counterclockwise direction.
Denote the continuous matrix of eigenvectors of $(A_\gamma, B_\gamma)$ at the beginning
of this loop as $V_0$ and that at the end of the loop as $V_1$.

Since the curve $\alpha$ does not contain any coalescing
point, the matrix $V_1$ would be the same as if, instead of
following the curve $\Gamma$, we were to follow $\Gamma_0$ from $y_0$ to $y_1$, then go from $y_1$ to
$y_0$ along $\alpha$, back from $y_0$ to $y_1$ along $\alpha$ in opposite direction
and then from $y_1$ to $y_0$ along $\Gamma_1$: 
$(\Gamma_0\cup \alpha)\cup((-\alpha)\cup\Gamma_1)$. Denote the
matrix of eigenvectors of $(A_\gamma, B_\gamma)$ at the end of the first loop 
$(\Gamma_0\cup \alpha)$ by $V_{\frac12}$.
Using the induction hypothesis along the closed
curve $\Gamma_0\cup \alpha$, we have
$$V_0=V_{\frac12}\hat D\,,$$
where $\hat D $ is a diagonal matrix $\hat D=\diag(\hat D_{11},\dots \hat D_{nn})$, with
$$\hat D_{11}=(-1)^{\hat d_1},\ \
\hat D_{ii}=(-1)^{\hat d_{i-1}+\hat d_i}\ \mathrm{for}\ i=2,\ldots,n-1,\ \
\hat D_{nn}=(-1)^{\hat d_{n-1}}\ $$
and $\hat d_i=d_i$, for all $i\ne j$, and $\hat d_j=d_j-1$.  Now, 
by looking at what happens on the second loop, by virtue of Theorem
\ref{nxncase}, we have that all columns of
$V_{\frac12}$ coincide with those of $V_1$, except for the $j$-th and $(j+1)$-st ones
which have changed in sign. Putting everything
together, we have $V_0=V_1D$ with $D$ as given in the statement of the Theorem.
\end{proof}

We do not study nongeneric coalescings, since they are not robust under perturbation;
see \cite{DiPu1} for considerations on these cases, for the symmetric eigenproblem.
With this in mind, in the final result we give we should think of all coalescings as
being generic CIs.  This theorem gives us a sufficient condition for the
existence of CIs inside a certain region.  This is the result on which we base our
numerical algorithm to detect CIs.

\begin{thm}\label{SuffCI}
Consider the pencil $(A,B)$, where
$A=A^T\in\cont^k(\Omega,\Rnxn)$, 
$B=B^T\in\cont^k(\Omega,\Rnxn)\succ 0$, and
let $\lambda_1(x)\ge\ldots\ge\lambda_n(x)$ be its continuous eigenvalues.
Let $\Gamma$ be a simple closed curve in $\Omega$
with no coalescing point for the eigenvalues on it, and let it be
parametrized as a $\cont^p$ ($p\ge 0$) function $\gamma$ in the variable $t$,
so that the function $\gamma:\ t\in \R \to \Omega$ is $\cont^p$ and $1$-periodic.
Let $m=\min(k,p)$ and let $A_\gamma$ and $B_\gamma$ be the $\cont^m$ restrictions of
$A$ and $B$ to $\gamma(t)$, and let $V$ diagonalize the pencil $(A_\gamma, B_\gamma)$.
Let $V_0=V(0)$ and $V_1=V(1)$, and define $D$ such that $V_0D=V_1$.

Next, let $2q$ be the
even\footnote{The reason for the even number of indices is that
$A_{\gamma}(t)V(t)=B_{\gamma}(t)V(t)\Lambda(t)$,
and $V(t+1)=V(t)\ D$. Since $V$ is continuous and invertible, then its determinant is
always positive or negative. But, since $V(1)=V(0)D$, then we must have
$\det(D)=1$}
number of indices $i_i$, $i_1<i_2<\cdots <i_{2q}$,
for which $D_{i_ii_i}=-1$. Let us
group these indices in pairs $(i_1,i_2),\dots,(i_{2q-1},i_{2q})$.
Then, $\lambda_i$ and $\lambda_{i+1}$ coalesced at least once
inside the region encircled by $\Gamma$, if 
$i_{2j-1}\le i < i_{2j}$ for some $j=1,\ldots,q$.
\end{thm}

\begin{rems}
Some comments are in order.
\begin{itemize}
\item[(i)]
Relative to generic CIs, suppose that from Theorem \ref{SuffCI} 
we have a $D$ with $D_{11}=-1=D_{44}$, all other $D_{ii}$'s being $1$.  Then,
we expect that inside the region encircled by $\Gamma$, the pairs
$(\lambda_1,\lambda_2)$, $(\lambda_2,\lambda_3)$, and $(\lambda_3,\lambda_4)$,
have coalesced.  Moreover, relative to generic CIs, and in the notation of
Theorem \ref{SuffCI}, we can say that there is an odd number of CIs points for
$\lambda_i$ and $\lambda_{i+1}$ inside the region encircled by $\Gamma$.
\item[(ii)]
Theorem \ref{SuffCI} cannot distinguish whether, inside $\Gamma$, some pair of eigenvalues
coalesced an even number of times or not at all.
\end{itemize}
\end{rems}

%\subsection{Tridiagonal pencils}
A final remark pertains to the case when $A$ and $B$ are both {\bf tridiagonal}.  This case is
quite difficult to handle for our algorithms of Section \ref{algos} that locate coalescing eigenvalues.
The reasons for the difficulties have been
already explained in our work on the symmetric eigenproblem; see the discussion on {\tt Veering} and
{\tt mingap} in \cite[Section 1.2]{DPP2}.  And, because of this,
in \cite[Section 2.3]{DPP2} we devised ad-hoc techniques for the tridiagonal case, techniques based upon 
the fact that for a symmetric tridiagonal matrix 
$A=\smat{a_1 & b_2 & & & \\ b_2 & a_2 & b_3 & & \\ & \sddots & \sddots & \sddots & \\
& & & b_n & a_n}$, a necessary condition to have repeated eigenvalues is that $b_i=0$, for some
$i$.  Unfortunately, in the case of a tridiagonal pencil $(A,B)$, there is no such simple necessary
condition that has to hold for having repeated eigenvalues.  For these reasons, the case of
$A$ and $B$ tridiagonal is left open for future study, and the results in Section \ref{NumExs}
do not include the tridiagonal case.

\vspace{.5cm}

\section{Algorithms to locate coalescing eigenvalues}\label{algos}

The procedure we implemented to locate coalescing  generalized eigenvalues is based on Theorem \ref{SuffCI}, 
and on the smooth generalized eigendecomposition 
$A(t)V(t)=B(t)V(t)\Lambda(t)$ along 1-d paths, as stated in Theorem \ref{Smoothness}.
Our goal is to obtain a sampling of these smooth $V$ and $\Lambda$ at some values of $t$.
Given a 1-parameter pencil $(A(t),B(t))$, for $t\in [0,1]$,  with $A=A^T\in\cont^k([0,1],\Rnxn)$, and
$B=B^T\in\cont^k([0,1],\Rnxn)\succ 0$,
we can assume that the eigenvalues are distinct for all $t\in [0,1]$, and 
$\lambda_1(t)>\lambda_2(t)>\ldots>\lambda_n(t)$. 

To compute $\Lambda={\rm diag}(\lambda_1,\lambda_2,\ldots,\lambda_n)$ and $V$   
we used a continuation procedure of predictor-corrector type, similar to the one 
%we adapted to this generalized eigenproblem case, the codes we
developed in \cite{DPP2} to obtain a sampling of the smooth ordered Schur decomposition
for symmetric 1-d functions.
For completness, we briefly describe here the step from a point $t_j$ to the new point $t_{j+1}$
of the new procedure, further remarking on the differences between the present procedure
and the one in \cite{DPP2}, to which we refer for a discussion of some algorithmic choices.

Given an ordered decomposition at $t_j$: $A(t_{j})V(t_j)= B(t_j) V(t_j) \Lambda(t_j)$
and a stepsize $h$,  we want the decomposition 
at $t_{j+1}=t_j+h$: $A(t_{j+1}) V(t_{j+1}) =B(t_{j+1}) V(t_{j+1}) \Lambda(t_{j+1})$, where
the factors $V(t_{j+1})$ and $\Lambda(t_{j+1})$ lie along the smooth path from $t_j$ to $t_{j+1}$. 
To get $\Lambda(t_{j+1})$ is easy to do with canned software, like {\tt eig} in {\tt Matlab},
since the eigenvalues are distinct, so we will keep them ordered.
Further, a $B$-orthogonal matrix $V_{j+1}$ such that $A(t_{j+1}) V_{j+1} =B(t_{j+1}) V_{j+1} \Lambda(t_{j+1})$ 
can be also obtained by standard linear algebra software, like the {\tt eig Matlab} command, and re-ordering. 
Then, recalling Corollary \ref{UniqueV}, we know that $V(t_{j+1})=V_{j+1}S$, where $S$ is a  
sign matrix, $S=\diag(s_1,\ldots, s_n)$, $s_i=\pm 1, i=1,\ldots,n$,
that is $V(t_{j+1})$ can be recovered by correcting the signs of the columns of $V_{j+1}$.  
Specifically, by enforcing minimum variation with respect to a suitably predicted factor $V^{pred}$, we set $S$
equal to the sign matrix   which minimizes $\| SV_{j+1}^TB(t_{j+1})V^{pred}-I\|_F$. 

Despite the overall simplicity of the basic step we just described,
if the stepsize $h$ is too large with respect to the variation of the factors, predicting the 
correct signs of the eigenvectors to follow the correct path may be a hard task.
This  difficulty is tipically encountered  when there is a pair  
of close eigenvalues, as happens  in presence 
of a {\tt veering} phenomenon. In this case smoothness could be mantained only 
by using very small stepsizes, being the variation of the eigenvectors inversely
proportional to the difference between eigenvalues (see the differential equations \eqref{DEforV}).  
Therefore we proceed in two different ways, depending
on the distance between consecutive eigenvalues. 
We say that a pair of eigenvalues ($\lambda_i$, $\lambda_{i+1}$) 
{\it is  close to veering} at $t_{j}+h$ if the following condition  holds:
\begin{equation*}%\label{vicini}
\frac{|\lambda_{i+1}(t_{j }+h)-\lambda_{i}(t_{j}+h)|}
{|\lambda_{i}(t_{j}+h)|+1}< \mathtt{toldist}\footnote{In our experiments we
have used $\mathtt{toldist}=10^6 \mathtt{eps}\approx 10^{-10}$}
% \,,  \quad
%\frac{|\lambda^{pred}_{i+1}-\lambda^{pred}_{i}|}
%{|\lambda^{pred}_{i}|+1}< \mathtt{toldist};
\end{equation*}
otherwise, the eigenvalues are considered well separated.  At the starting point $t_j$,
eigenvalues are assumed to be well separated.\\[1ex]
{\it Case 1. Some pair of eigenvalues is close to veering at} $t_{j}+h$.\\
In practice during a veering  close eigenvalues may become numerically undistinguishable,
and the corresponding $B$-orthogonal eigenvectors change very rapidly  
within a very small interval, out of which the eigenvalues are again well separated. 
To overcome this critical veering zone, we proceed 
%as in \cite{DGP} and \cite{DPP2}
by  computing a  smooth block-diagonal eigendecomposition (see Theorem \ref{BlockDiag}):
\begin{equation} \label{BlockDiagV}
V_B^T(t)A(t)V_B(t)=\Lambda_B(t)=\diag(\Lambda_1(t),,...,\Lambda_p(t)),
  ~~~~ V_B^T(t)B(t)V_B(t)=I,\;~~~~ t\ge t_j , 
\end{equation}
where  close eigenvalues are grouped into one block, 
so that the eigenvalues of each $\Lambda_i$ are well separated from the others.  
We do not expect, nor consider, the nongeneric case of three
or more close eigenvalues, hence each $\Lambda_i(t)$  is either an eigenvalue or a $ 2\times 2$ block.
Using the Cholesky factorization of $B=LL^T$,  we first re-write \eqref{BlockDiagV} as follows:
\begin{equation*}  
\underbrace{  V_B^T L}_{Q_B^T} \,\underbrace{ L^{-1} A L^{-T}}_{\tilde A}\, \underbrace{ L^T V_B}_{Q_B} =\Lambda_B , 
\quad\quad
   ~~~~ \underbrace{V_B^T L}_{Q_B^T}  \, \underbrace{  L^T V_B}_{Q_B} =I,\; \quad t\ge t_j. 
 \end{equation*}
Then to compute the smooth orthogonal transformation $Q_B$  and block-diagonal $\Lambda_B$,
we  use a procedure for the continuation of invariant subspaces, 
which is based on Riccati transformations (see \cite{DP} and \cite{DPP2} for details of this technique).
Starting at $t_j$, we continue with this standard  block eigendecomposition  
until all eigenvalues are again well separated; this
happens at some value $t_f$, and we set $t_{j+1}=t_f$. Then 
$V_B(t_{j+1})=L^{-T}(t_{j+1}) Q_B(t_{j+1})$.\\
A key issue is how to recover the complete smooth eigendecomposition at $t_{j+1}$. 
Indeed Theorem \ref{BlockDiag} guarantees the existence of  decomposition \eqref{BlockDiagV} but not its uniqueness,
as can be easily verified by rotating the columns of $V_B$ - or $Q_B$ - corresponding to a $2 \times 2$ diagonal block. 
In \cite{DPP2}, to which we refer for the details, we show 
how these subspaces  can be rotated  to obtain an accurate predicted factor $V^{pred}$ 
which allows to correct the signs of $V_{j+1}$'s columns, 
and continue the complete smooth  eigen-decompositon at $t_{j+1}+h$. \\[1ex]
{\it Case 2. All eigenvalues are well separated}.\\
In this case, through our predictor-corrector strategy the stepsize is adapted based on both 
eigenvalues and eigenvectors variations. 
The following variation parameters 
\begin{equation}\label{rholv}\begin{split}
 \rho_{\lambda} & = \max_i   \frac{|\lambda_{i}(t_{j+1})-\lambda_{i}^{pred}|} {|\lambda_{i}(t_{j+1})|+1} \quad {\rm and} \\
\rho_V & =\frac{ \bigl( {\rm{tr}}\left[ (V(t_{j+1})-V^{pred})^TB(t_{j+1})(V(t_{j+1})-V^{pred}) \right] \bigr)^{1/2}}{\sqrt{n}}
\end{split}\end{equation}
are used  both  to update the stepsize and to   accept or reject a step  (see Steps 4 and 5 in 
Algorithm \ref{sfs} below).
Accurate predictors  are hence mandatory for the efficiency of the overall  procedure.  
We obtain them by taking an Euler step: 
$$\Lambda^{pred}=\Lambda(t_j)+h\,  {\dot \Lambda}_j , \qquad   V^{pred}=V(t_j)+h \, {\dot V}_j,$$
 in the differential equations \eqref{DEforV}, where the  derivatives  $\dot \Lambda(t_j)\simeq  \dot \Lambda_j$ 
 and  $\dot V(t_j)\simeq \dot V_j$  are approximated
by replacing ${\dot A}(t_j)$ with $(A(t_{j+1})-A(t_j))/h$  and ${\dot B}(t_j)$ with $(B(t_{j+1})-B(t_j))/h$. 
Setting  $A_V= V(t_j)^T A(t_{j+1}) V(t_j)$ and   $B_V=V(t_j)^T B(t_{j+1}) V(t_j)$ we have
\begin{equation}\label{predittori}
\Lambda^{pred}= \diag(A_V)- \Lambda(t_j) (\diag (B_V) -I), \quad 
V^{pred}=V(t_j)(I+P+H),\end{equation}
 where $P=(I-B_V)/2$, and  $H$ is the skew-symmetric matrix such that
\medskip
$$H_{ik}=-  H_{ki}=\frac{[A_V]_{ik}}{\lambda_i-\lambda_k} - 
\frac{\lambda_i+\lambda_k}{\lambda_i-\lambda_k} \,\frac{ [B_V]_{ik}}{2}, \;\; {\rm for} \; i<k,
\;\;\;\;\;\;H_{ii}=0\;\; {\rm for} \;i=1,\dots,n.$$
We remark that $\rho_V$, $\Lambda^{pred}$ and $V^{pred}$ reduce to the corresponding 
quantities we used in \cite{DPP2} %to control the variation of the eigenvectors 
for the standard symmetric eigenvalue problem, where $B=I$ and $V$  is orthogonal.

Further, after we reached $t_{j+1}$ with a successful step, we compute 
the following  predicted eigenvalues of secant type: 
\begin{equation*} \label{secant}
 \lambda_i^{sec}=\lambda_i(t_{j+1})+h \dot{\lambda}_i^{sec}= \lambda_i(t_{j+1})+
h \frac{ \lambda_i(t_{j+1})-\lambda_i(t_j)} { t_{j+1}- t_j}, \;\;\; i=1,\ldots,n.
\end{equation*}
Step 7 in Algorithm \ref{sfs} uses these secant prediction, and its rationale is
that if $\lambda^{sec}_i<\lambda^{sec}_{i+1}$ for some $i$ (recall that we have to
obtain $\lambda_i>\lambda_{i+1}$), then 
the new step $t_{j+1}+h$ is likely to fail, and therefore $h$ will be safely reduced as in \eqref{reduceh}.

\algo{sfs}{Predictor-Corrector step for well separated eigenvalues} {
Input:  $t_j$,  an ordered  decomposition  $A(t_{j})V(t_j)= B(t_j) V(t_j) \Lambda(t_j)$, a  stepsize $h$.\\[1ex]
Output:  $t_{j+1}$,  the smooth  decomposition  $A(t_{j+1})V(t_{j+1})= B(t_{j+1}) V(t_{j+1})   \Lambda(t_{j+1})$,\\ 
\hspace*{1.4cm}
an updated stepsize $h$.\\[1ex]
1. Set $t_{j+1}=t_j+h$,  and compute ~~$V^{pred}$ ~and~ 
 $\Lambda^{pred}$ by \eqref{predittori};\\
2. Compute an  ordered decomposition 
$A(t_{j+1})V_{j+1}= B(t_{j+1}) V_{j+1}   \Lambda(t_{j+1})$; \\
3.  Find the sign  matrix $S$ to correct eigenvectors, and set $V(t_{j+1}) = V_{j+1} S$;\\
4.   Compute   $\rho_{\lambda}$ and $\rho_V$ from \eqref{rholv}, set $\rho=
 \max \{\rho_{\lambda}, \rho_V \}/{\tt tolstep}$ (in our experiments, we
have used $\mathtt{tolstep}=  10^{-2}$), 
and update $h=h/\rho;$%, where:
\vskip 2pt
5. If $\rho\le 1.5$, accept the step; \\
\hspace*{10pt} otherwise declare {\it failure}, go to step 1,  and retry  with the new (smaller) $h$;\\
6.   Compute   $\displaystyle  \dot{\lambda}_i^{sec} =   \frac{ \lambda_i(t_{j+1})-\lambda_i(t_j)} { t_{j+1}- t_j}\;$
~ and ~
$\;\lambda_i^{sec} = \lambda_i(t_{j+1})+h    \dot{\lambda}_i^{sec},$ ~ for $i=1,\ldots,n;$    \\
7. If $\lambda^{sec}_i<\lambda^{sec}_{i+1}$ for some $i=1,\ldots,n-1$, reduce $h$ as in
\begin{equation}\label{reduceh}
 h=0.9 \, h \min \left \{ 
 \frac{ \lambda_i(t_{j+1})-\lambda_{i+1}(t_{j+1})} {\dot{\lambda}^{sec}_{i+1}-\dot{\lambda}^{sec}_{i}}, ~
{\rm for} ~\lambda^{sec}_i<\lambda^{sec}_{i+1}, ~i=1,\dots, n-1 \right \}.
\end{equation}
}

\begin{rem}\label{WhyNotCholAndQR}
Observe that smooth factors $\Lambda$ and $V$ could be  obtained also
via a smooth Schur decomposition $\tilde A(t)=Q(t)\Lambda(t) Q^T(t)$ of the symmetric matrix
$\tilde A=L^{-1}AL^{-T}$,  with  $B=LL^T$, by using the procedure developed in \cite{DPP2}
and setting $V=L^TQ$.
However, Algorithm \ref{sfs} which is tailored to the original generalized eigenproblem results
much more efficient, for general $A$ and $B$;
e.g., in our experiments in {\tt Matlab} the main cost is clearly in step 2 of the algorithm,
which we resolve with a call to {\tt eig}, and using {\tt eig}$(A,B)$ costs less than half of
the execution time encountered forming $\tilde A$ and using {\tt eig}$(\tilde A)$.
\end{rem}
 
\vspace{.5cm}

\section{Random matrix ensemble and experiments}\label{NumExs}

Making use of the algorithms presented in Section \ref{algos}, we have performed a numerical study on
coalescence of eigenvalues for parametric eigenproblems of interest in computational mechanics. 
In this Section we report on these experiments.

We considered pencils $A-\lambda B$ whose matrices belong to a random matrix ensemble called ``SG$^+$''. 
This ensemble has been introduced in \cite{Soize3} for modelling uncertainties in computational mechanics. 
Matrices in this ensemble are characterized by a property called \emph{dispersion}, which is controlled by 
a dispersion parameter $\delta$ that must satisfy $0<\delta<\sqrt{(n+1)(n+5)^{-1}}$. Moreover we 
contemplated also banded SG$^+$ matrices, i.e. matrices in SG$^+$ ``truncated" so to have bandwidth 
$b=1,\ldots,n-1$, where $b=1$ means tridiagonal and $b=n-1$ means ``full''. Our goal is to investigate 
the effect of bandwidth and dispersion on how the number of conical intersections varies as we increase 
the dimension of the matrices.

We now illustrate in detail how our random matrix functions are defined. 
First, given integers $n\ge 2$ and $b=1,\ldots,n-1$, and dispersion parameter $0<\delta<\sqrt{(n+1)(n+5)^{-1}}$, 
we construct the following $n\times n$ matrices (only the non-zero entries are explicitly defined): 

\medskip
\begin{algorithmic}
\For {\text{all} $i,j$ such that $0<i-j\le b$ and $k=1,2,3,4$}
	\State $(L_{A,k})_{ij}\gets \sigma_n u$, where $\sigma_n=\delta(n+1)^{-1/2}$ and $u\in{N(0,1)}$
	\State $(L_{B,k})_{ij}\gets \sigma_n u$, with $\sigma_n$ and $u$ as above
\EndFor
\For {$i=1,\ldots,n$}
	\State $(D_A)_{ii}\gets \sigma_n\sqrt{2v_i}$, where $\sigma_n=\delta(n+1)^{-1/2}$, $v_i\in\Gamma(a_i,1)$ 
	and $a_i=\frac{n+1}{2\delta^2}+\frac{1-i}{2}$
	\State $(D_B)_{ii}\gets \sigma_n\sqrt{2v_i}$, with $\sigma_n$ and $v_i$ as above
\EndFor
\end{algorithmic}
\medskip

\noindent where $N(0,1)$ is the normal distribution with zero mean and variance 1, while $\Gamma(a_i,1)$ 
is the gamma distribution with shape $a_i$ and rate 1. Then, for all $(x,y)$ in $\R^2$, we define the 
following matrix functions:

\begin{equation*}
\begin{split}
 L_A(x,y)\coloneqq &\,\cos(x)L_{A,1}+\sin(x)L_{A,2}+\cos(y)L_{A,3}+\sin(y)L_{A,4}+D_A,\\
 L_B(x,y)\coloneqq &\,\cos(x)L_{B,1}+\sin(x)L_{B,2}+\cos(y)L_{B,3}+\sin(y)L_{B,4}+D_B, \\
 A(x,y)\coloneqq &\,L_A(x,y) {L_A(x,y)}^T, \\
 B(x,y)\coloneqq &\,L_B(x,y) {L_B(x,y)}^T.
\end{split}
\end{equation*}

We point out that:
\begin{enumerate}[a)]
 \item all matrices $L_{A,k}$ and $L_{B,k}$ are strictly lower triangular and have bandwidth $b$, while 
 $D_A$ and $D_B$ are diagonal; therefore, both $A(x,y)$ and $B(x,y)$ have bandwidth $b$;
\item $A(x,y)=A(x,y)^T$ and $B(x,y)=B(x,y)^T$ are positive definite for all $(x,y)$;
\item the nontrivial entries of all matrices $L_{A,k}$, $L_{B,k}$, $D_A$ and $D_B$ are independent 
random variables; more precisely: $A(x,y)\in$ SG$^+$ for all $(x,y)$, and the same is true for $B(x,y)$; 
 \item the probability density function of the diagonal entries of $D_A$ and $D_B$ matrix depends on 
 their position along the diagonal.
\end{enumerate}

In our experiments, we have fixed five values of the dispersion parameter: 
$\delta=0.05$, $0.25$, $0.45$, $0.65$, $0.85$, and considered four possible bandwidths 
$b=3,4,5,\mathrm{full}$. For each combination of $\delta$ and $b$, and for dimensions 
$n=50, 60, \ldots, 120$, we have constructed 10 realizations of matrix pencils 
$A(x,y)-\lambda B(x,y)$ and performed a search for conical intersections for the pencil 
over the domain $\Omega=[0, \pi]\times [0,2\pi]$.
The detection strategy consisted of subdividing the domain $\Omega$ into $64\times 128$
square boxes and computing a smooth generalized eigendecomposition of the pencil around 
the perimeter of each box. The presence of conical intersections inside each box is betrayed
by sign changes of the columns of the smooth $B$-orthogonal matrix that diagonalizes 
the pencil, see Theorem \ref{SuffCI} and the subsequent remarks.

Our pourpose was to fit the data with a power law
\begin{equation}\label{eq:powerlaw}
 \# \mathrm{\ of\ CIs} = c\, \mathrm{dimension\,}^p,
\end{equation}
averaging the number of conical intersections over the 10 realizations. 
The outcome of the experiments is illustrated in Figure \ref{fig:bestfit} and 
Table \ref{tab:data}. Figure \ref{fig:bestfit} shows the superposition of the 
20 linear regression lines obtained by computing a least squares best fit over the 
logarithm of the data, so that $p$ and $c$ in \eqref{eq:powerlaw} represent, respectively, 
slope and intercept of the lines. The Figure reveals 4 groups of 5 lines each, where lines 
in the same group share the value of the bandwidth. It is evident that the dispersion parameter 
$\delta$ has a negligible effect on the exponent $p$, and mostly also on the factor $c$ 
(with the exception of the full bandwidth case). In contrast, bandwidth has a significant effect 
on the exponent $p$, that increases from $p\approx 2$ of the full bandwidth case to 
$p\approx 2.6$ of the heptadiagonal case. A similar study was conducted in \cite{DPP2} for 
the GOE (Gaussian Orthogonal Ensemble) model. A look at Table 4 in that work shows, for the GOE 
model, a faster growth of the exponent $p$ as the bandwidth is decreased, compared to what we have 
observed here for the SG$^+$ model. For convenience, below we report the values of $p$ obtained for 
the two models (for the SG$^+$ model we average over all values of $\delta$, since variations for 
different values of $\delta$ are negligible):

\begin{center}
\medskip
 \begin{tabular}{c|c|c}
 bandwidth & SG$^+$ & GOE \\ \hline
full & 2.01 & 2.00 \\
5 & 2.46 & 2.55 \\
4 & 2.54 & 2.66 \\
3 & 2.60 & 2.73 \\
 \end{tabular}
 \medskip
\end{center}

Table \ref{tab:data} shows, as an example, a synopsis of the outcome of our experiments for the case $\delta=0.45$. 
The Table reports on the number of conical intersections that have been detected and on the results of the least 
squares best fits.

Finally, in Figure \ref{fig:elaptime} we give an account of the computational time required by our experiments. 
The barplot indicates, for all values of bandwidth and dimension we have considered, the average elapsed time 
of each computation (averaged over all realizations and all values of $\delta$). The data are normalized 
with respect to the computation that required the longest time, which corresponds to the heptadiagonal 
case and largest dimension $n=120$ and was about 12.5 hours. By contrast, the fastest computation 
corresponds to the full bandwidth case and smallest dimension $n=50$ and was about 15 minutes. 
The figure clearly indicates that the computational effort is directly proportional to the number 
of conical intersections. This fact comes with no surprise as, in the vicinity of each conical 
intersection, the eigenvectors exhibit rapid variations that require severe restrictions on the 
stepsize for our continuation algorithm.

We were not able to perform experiments for the tridiagonal and pentadiagonal cases. 
This is due to the fact that, as the bandwidth gets critically small and the dimension 
sufficiently large, a significant amount of sharp variations of the eigenvectors occur 
within intervals of size smaller than machine precision, ruling out our (but, actually, any) 
numerical continuation solver. These difficulties where already encountered (and explained) 
in \cite{DPP2} (see Remark 3.1 therein). See also the considerations at the end of Section 
\ref{CIs} of this work.

\begin{table}
\caption{The table shows (left to right): bandwidth, dimension of the problem, 
number of conical intersections detected (average over 10 realizations from the SG$^+$ ensemble), 
outcome of the log-log linear least squares regression (including the root mean square deviation). 
The data refer to the case $\delta=0.45$.}
\label{tab:data}
\begin{tabular}{|c||c|c||c|}
\cline{1-4}
{\rm bandwidth} & {\rm dimension} & {\rm avg \# CIs} & {\rm power law} \\
\cline{1-4}
\multirow{8}{*}{3}
& 50 & 3340 & \multirow{3}{*}{$p=2.59$} \\
& 60 & 5318 & \\
& 70 & 7882 & \multirow{4}{*}{$c=0.13$} \\
& 80 & 11269 & \\
& 90 & 15155 & \\
& 100 & 20100 & \multirow{3}{*}{rmsd = 4.59$\times 10^{-3}$} \\
& 110 & 25460 & \\
& 120 & 32013 & \\
\cline{1-4}
\multirow{8}{*}{4}
& 50 & 2651 & \multirow{3}{*}{$p=2.55$} \\
& 60 & 4312 & \\
& 70 & 6304 & \multirow{4}{*}{$c=0.12$} \\
& 80 & 8803 & \\
& 90 & 11964 & \\
& 100 & 15350 & \multirow{3}{*}{rmsd = 1.04$\times 10^{-2}$} \\
& 110 & 20010 & \\
& 120 & 25269 & \\
\cline{1-4}
\multirow{8}{*}{5}
& 50 & 2456 & \multirow{3}{*}{$p=2.45$} \\
& 60 & 3785 & \\
& 70 & 5508 & \multirow{4}{*}{$c=0.17$} \\
& 80 & 7679 & \\
& 90 & 10040 & \\
& 100 & 13128 & \multirow{3}{*}{rmsd = 1.07$\times 10^{-2}$} \\
& 110 & 17010 & \\
& 120 & 20797 & \\
\cline{1-4}
\multirow{8}{*}{full}
& 50 & 1916 & \multirow{3}{*}{$p=2.02$} \\
& 60 & 2827 & \\
& 70 & 3799 & \multirow{4}{*}{$c=0.71$} \\
& 80 & 4981 & \\
& 90 & 6422 & \\
& 100 & 7853 & \multirow{3}{*}{rmsd = 7.34$\times 10^{-3}$} \\
& 110 & 9517 & \\
& 120 & 11282 & \\
\cline{1-4}
\end{tabular}
\end{table}

\begin{figure}[h]
\begin{center}
\includegraphics[width=.8\textwidth]{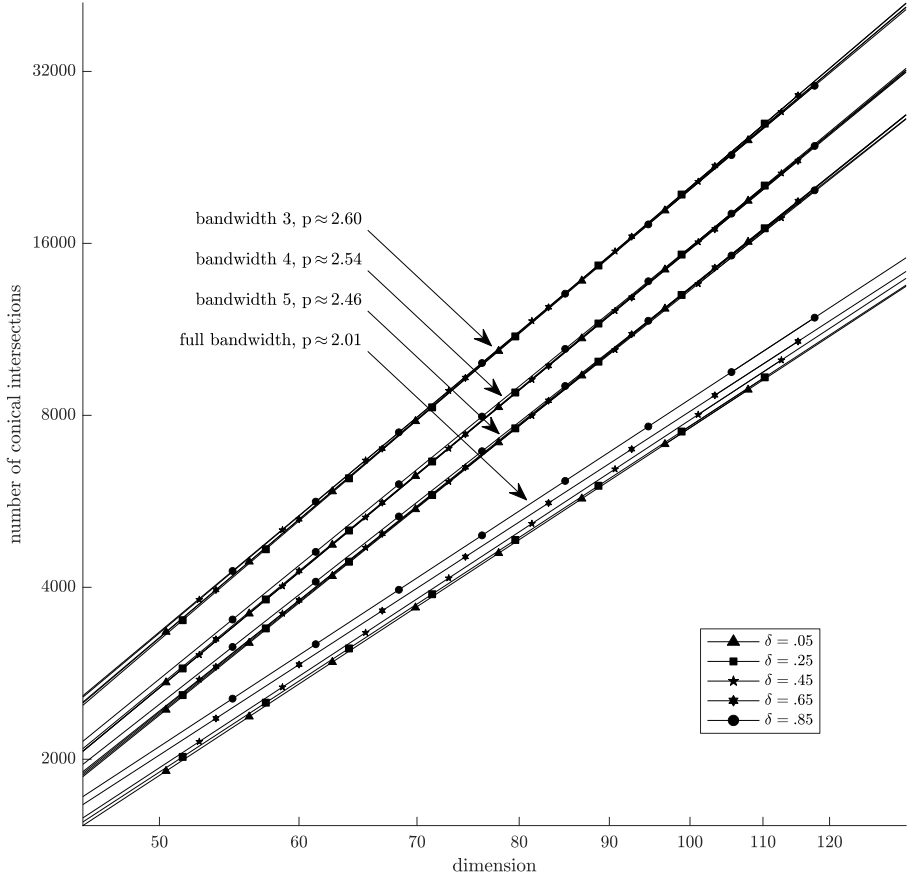}
\caption{For each value of $b=3, 4, 5, \mathrm{full}$ and $\delta=0.05, 0.25, 0.45, 0.65, 0.85$, 
we have performed a log-log linear least squares regression. The Figure shows the best fit lines 
for all combinations of $b$ and $\delta$, and also the average exponent (slope after the log-log 
transformation) $p$ for each value of $b$.}
\label{fig:bestfit}
\end{center}
\end{figure}

\begin{figure}[h]
\begin{center}
\includegraphics[width=.666\textwidth]{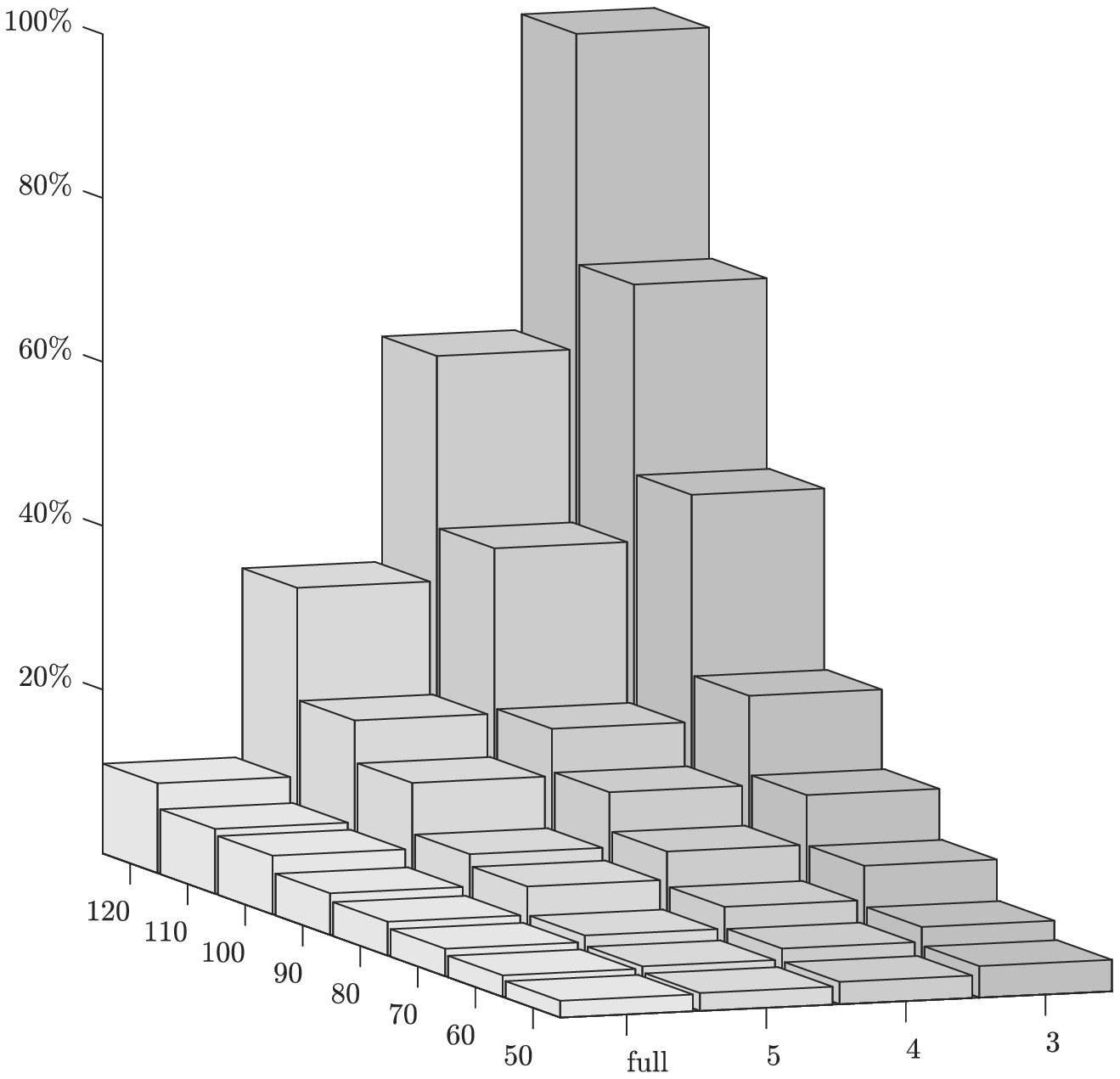}
\caption{default}
\label{fig:elaptime}
\end{center}
\end{figure}

All computations have been performed on the ``Partnership for an Advanced Computing Environment'' 
(PACE), the high performance computing infrastructure at the Georgia Institute of Technology, Atlanta, Georgia, USA.

\vspace{.5cm}

% \section{Experimental results}\label{NumResults}
% 
% 
% \vspace{.5cm}
% 
\section{Conclusions}\label{Conclusions}
In this work we have considered symmetric positive definite pencils, 
$A(x)-\lambda B(x)$,
where $A$ and $B$ are symmetric matrix valued functions in $\Rnxn$,
smoothly depending on parameters $x$, and 
$B$ is also positive definite.  We gave general smoothness results for the
one parameter and two-parameter case, and gave results to characterize the
(generic) case of coalescing eigenvalues in the two-parameter case, the so-called
conical intersections.
We further presented, justified, and implemented, new algorithms to locate
parameter values where there are conical intersections.  These algorithms were
used to perform a statistical study of the number of conical intersections for
pencils of several bandwidths.

Several issues are still requiring a more ad-hoc study.  For example,
the case of both $A$ and $B$ tridiagonal (e.g., see \cite{Wilkinson}) is still
not resolved in a satisfactory way for the generalized eigenvalue problem, and
perhaps the technique of \cite{LLZ} can be adapted to the parameter 
dependent case examined by us.  But also other problems remain to be examined,
especially from the algorithmic point of view, like the case of large number of equal
eigenvalues seen in some structural engineering works (e.g., see \cite{SWF:WavePerLatt}).


\vspace{.5cm}

\begin{thebibliography}{10}

\bibitem{ChernDieci}
{\sc J.L. Chern and L. Dieci},
\newblock Smoothness and periodicity of some matrix decompositions.
\newblock {\em {SIAM J}. {M}atrix {A}nal. {A}ppl.}, 22:772--792, 2001.

\bibitem{DieciEirola}
{\sc L.~Dieci and T.~Eirola},
\newblock On smooth decomposition of matrices.
\newblock {\em {SIAM J}. {M}atrix {A}nal. {A}ppl.}, 20 (1999), pp.~800--819.

\bibitem{DP}
{\sc L.~Dieci and A.~Papini},
\newblock Continuation of eigendecompositions.
\newblock {\em Future Generation Computer Systems}, 
19, pp. 1125-1137, 2003.

%\bibitem{DPP1}
%{\sc L.~Dieci, A.~Papini and A.~Pugliese},
%\newblock Approximating coalescing points for eigenvalues of
%Hermitian matrices of three parameters.
%\newblock {\em {SIAM J}. {M}atrix {A}nal. {A}ppl.}, 
%34-2, pp. 519-541, 2013.

\bibitem{DPP2}
{\sc L.~Dieci, A.~Papini and A.~Pugliese},
\newblock Coalescing points for eigenvalues of banded matrices
depending on parameters
with application to banded random matrix functions.
\newblock {\em {Numerical Algorithms}}, 
pp. 1-26, 2018.

\bibitem{DiPu1}
{\sc L.~Dieci and A. Pugliese},
\newblock Two-parameter SVD: Coalescing singular values and periodicity.
\newblock {\em {SIAM J}. {M}atrix {A}nal. {A}ppl.}, 31:375--403, 2009.

\bibitem{Gingold}
{\sc H.~Gingold}.
\newblock A method of global blockdiagonalization for matrix-valued functions.
\newblock {\em {SIAM J. Math. Anal.}}, 9-6:1076--1082, 1978.

\bibitem{vanHemmenAndo}
{\sc J. L. van Hemmen and T. Ando},
\newblock An Inequality for Trace Ideals.
\newblock {\em Commun. Math. Phys.}, 76, pp. 143-148, 1980.

\bibitem{Hirsch:DiffTop}
{\sc M.W. Hirsch}.
\newblock {\em Differential {T}opology}.
\newblock Springer-Verlag, New--York, 1976.

\bibitem{HsiehSibuya}
{\sc P.F. Hsieh and Y.~Sibuya}.
\newblock A global analysis of matrices of functions of several variables.
\newblock {\em {J}. {M}ath. {A}nal. {A}ppl.}, 14:332--340, 1966.

\bibitem{Kato}
{\sc T.~Kato},
 {\em Perturbation Theory for Linear Operators},
 2nd ed., Springer-Verlag, Berlin, 1976.
 
 \bibitem{LLZ}
 {\sc K. Li, T-Y Li, and Z. Zeng},
 \newblock An algorithm for the generalized symmetric tridiagonal eigenvalue problem.
 \newblock {\em Numerical Algorithms}, 8, pp. 269–-291, 1994.
 
\bibitem{SWF:WavePerLatt}
{\sc A. Srikantha Phani, J. Woodhouse and N.A. Fleck}, 
\newblock Wave propagation in two-dimensional periodic lattices. 
\newblock{{J. Acoust. Soc. Am.}}, 119, pp. 1995--2005, 2006.

\bibitem{Soize3}
{\sc C. Soize},
\newblock Random matrix models and nonparametric method for uncertainty quantification.
\newblock Handbook for Uncertainty Quantification, 1.
R. Ghanem, D. Higdon, and H. Owhadi Editors. Springer International
Publishing Switzerland, pp.219-287, 2017.

\bibitem{Tisseur}
{\sc F. Tisseur and K. Meerbergen},
\newblock The quadratic eigenvalue problem.
\newblock {\em SIAM Review}, 43-2, pp. 235-286 (2001).

\bibitem{Wilkinson}
{\sc J.H. Wilkinson},
\newblock {\it Algebraic eigenvalue problem}. \newblock
Clarendon Press. Oxford University Press.  Oxford, UK, 1988.



\end{thebibliography}
\end{document}